\documentclass[11pt,leqno]{amsart}
\usepackage[a4paper,margin=27mm,headheight=14pt]{geometry}
\usepackage[T1]{fontenc}
\usepackage{lmodern,amssymb,mathrsfs}
\usepackage{mathtools,amsthm,amscd,booktabs,array,longtable,microtype}
\usepackage{enumitem}
\usepackage[hidelinks,unicode]{hyperref}
\usepackage{bookmark}
\setlist[enumerate]{label=\textup{(\alph*)},leftmargin=*,itemsep=3pt,topsep=5pt}
\newtheorem{theorem}{Theorem}[section]
\newtheorem{lemma}[theorem]{Lemma}
\newtheorem{proposition}[theorem]{Proposition}
\newtheorem{corollary}[theorem]{Corollary}
\theoremstyle{definition}
\newtheorem{definition}[theorem]{Definition}

\theoremstyle{remark}

\newcommand{\CC}{\mathbb C}
\newcommand{\FF}{\mathbb F}

\newcommand{\cP}{\mathcal P}
\newcommand{\cR}{\mathcal R}
\newcommand{\cJ}{\mathcal J}
\newcommand{\cA}{\mathcal A}
\newcommand{\cO}{\mathfrak O}
\newcommand{\one}{\mathbf 1}
\newcommand{\eps}{\varepsilon}
\newcommand{\ind}{\mathrm{in}}
\newcommand{\out}{\mathrm{out}}
\newcommand{\Iroot}{\mathtt I}
\newcommand{\Oroot}{\mathtt O}
\DeclareMathOperator{\Spec}{Spec}
\DeclareMathOperator{\ord}{ord}
\DeclareMathOperator{\TC}{TC}
\DeclareMathOperator{\occ}{occ}

\DeclareMathOperator{\sgn}{sgn}
\allowdisplaybreaks[2]
\title[Capacities and phase repair]{Zero forcing on maximal outerplanar graphs:\\capacities and phase repair}
\author[Z. Han]{Zongjian Han}
\author[Y. Liang]{Yiran Liang}
\thanks{\raggedright\textsc{Zongjian Han.} School of Mathematics and Statistics, Lanzhou University, Lanzhou, Gansu 730000, China.\enspace\textit{E-mail:} \texttt{dblnhzj@163.com}}
\thanks{\raggedright\textsc{Yiran Liang.} School of Mathematical Sciences, Nankai University, Tianjin, China.\enspace\textit{E-mail:} \texttt{cldlllyr@163.com}}
\date{}
\hypersetup{pdftitle={Zero forcing on maximal outerplanar graphs: capacities and phase repair},pdfauthor={Zongjian Han and Yiran Liang}}
\begin{document}
\raggedbottom
\begin{abstract}
We give a formula for the zero forcing number of a maximal biconnected
outerplanar graph. If its weak dual is a path, then $P(G)=Z(G)=2$.
Otherwise the branching faces determine a canonical core $C$ and a family
of terminal strips. We prove that an optimal forcing process can be
chosen to omit the terminal root edges and to satisfy explicit vertex
and face capacities on $C$. Fix full forests on the terminal strips.
For an ear allocation $\theta$ and an allowed core selection $J$, let
$\tau_C(\theta,J)$ be the least number of root connections whose removal
makes the resulting forest executable. Then
\[
Z(G)=|V(C)|-\max_{\theta,J}\bigl(|J|-\tau_C(\theta,J)\bigr).
\]
The path covering number is obtained by maximizing $|J|$ alone.
We characterize executable directions by equations associated with
blocking circuits. The two-path equations define an obstruction in the
first cohomology of a signed contact multigraph. Local replacements and
a finite comparison of boundary recurrences give the required normal form.
A 56-state recurrence constructs a minimum zero forcing set and a full
chronology in linear time from indexed triangle adjacency. We also give
exact replacements of rooted blocks and infinite families for which
$P(G)<Z(G)$.
\end{abstract}
\maketitle
\setcounter{tocdepth}{1}
\tableofcontents
\clearpage
\section{Introduction}
A zero-forcing process starts from black vertices and uses the rule that a
black vertex with exactly one white neighbour may color that neighbour black.
Its minimum initial-set size is $Z(G)$. Each completed process decomposes the
vertices into directed induced paths, one starting at each initial vertex.
If $P(G)$ denotes the minimum number of vertex-disjoint induced paths covering
$G$, then
\[
 P(G)\le Z(G).
\]
The difference $Z(G)-P(G)$ measures the obstruction to executing a path
cover. Even after its paths are oriented, the prescribed forces may
require several paths to wait for one another.

For a maximal biconnected outerplanar graph, we describe this obstruction
in terms of its polygon triangulation. We first put an optimal forcing
process into a normal form. Its core edges satisfy local degree and face
conditions, and its terminal strips use fixed forests with selected root
connections removed. The remaining direction problem is expressed by
blocking circuits among the covering paths. These constructions give
formulas for both $P(G)$ and $Z(G)$.

\subsection{The core and its capacities}
Such a graph is a triangulated polygon. Its weak dual is the tree with one
vertex per bounded face and one edge for two faces sharing a diagonal. If this
tree is a path, the two-root construction in Section~2 gives
$P(G)=Z(G)=2$. Otherwise let $C$ be the union of the faces in the smallest
subtree containing all degree-three vertices. Every remaining component is a
chain of triangles attached to an outer side of $C$; we call it a terminal
strip.

A strip with at least two triangles is a fan if one root is adjacent to every
interior vertex, and a serpent otherwise. In a fan, call that root its key
and the other root its nonkey. A single triangle is an ear. Fix the full
forests as in Definition~\ref{def:strip} and the construction following it.
A fan requires one root connection at its nonkey;
a serpent requires one at each root. For every ear choose one demanded root,
and write $\theta$ for these choices. If $t_v$ demanded connections meet a
core vertex $v$, its remaining capacity is
\[
 c_v=2-t_v\in\{0,1,2\}.
\]
Let $\mathcal R$ be the terminal root edges. The allowed core selections are
\[
 \mathcal J(C,\theta)=
 \left\{J\subseteq E(C)\setminus\mathcal R:
 \deg_J(v)\le c_v,
 \ |J\cap E(\Delta)|\le1\text{ for every core triangle }\Delta\right\}.
\]
In particular, a shared key of capacity two may use two core edges. A bare
outer side of the core remains eligible.

Theorem~\ref{thm:hard-general} proves that an optimal forcing process can
be chosen with $J\in\mathcal J(C,\theta)$. It omits every terminal root
edge and uses the fixed terminal forests with some root-to-interior edges
removed.

\subsection{Exact path and forcing formulas}
Put $r=|V(C)|$. For $J\in\mathcal J(C,\theta)$, the union of $J$ and the full
terminal forests is an induced path forest with $r-|J|$ components. Let
$\mathcal A_\theta$ be its selected root-to-interior connections. For
$D\subseteq\mathcal A_\theta$, cutting $D$ gives $r-|J|+|D|$ paths. Define
\[
 \tau_C(\theta,J)=
 \min\{|D|:D\subseteq\mathcal A_\theta,
 \text{ the resulting paths admit an executable orientation}\}.
\]
The minimum of an empty set is $+\infty$.
Theorem~\ref{thm:hard-general} gives an allowed configuration with finite
repair cost. The blocking equations of Section~\ref{sec:phases} determine
which orientations are executable. We obtain
(Theorems~\ref{thm:path-capacity} and~\ref{thm:main})
\begin{equation}\label{eq:intro-main}
 \begin{aligned}
 P(G)&=r-\nu(C),&
 \nu(C)&=\max_{\theta,J\in\mathcal J(C,\theta)}|J|,\\
 Z(G)&=r-\max_{\theta,J\in\mathcal J(C,\theta)}
                 \bigl(|J|-\tau_C(\theta,J)\bigr).
 \end{aligned}
\end{equation}
Consequently,
\[
 Z(G)-P(G)=
 \min_{\theta,J}\bigl(\nu(C)-|J|+\tau_C(\theta,J)\bigr).
\]
The first term is the cost of foregoing core connections; the second is the
cost of making the retained connections executable. Equality $Z(G)=P(G)$
holds precisely when some maximum capacity-allowed selection has an
executable full forest.

To encode the direction problem, give each nontrivial path a reference
orientation and a sign $s_i$ with $s_i^2=1$. A blocking circuit traverses
nonempty segments of distinct paths and joins them by graph edges. If its
segment signs are $\sigma_j$, its two cyclic orientations are excluded by
\[
 \prod_j\frac{1+\sigma_js_{i_j}}2=0,
 \qquad
 \prod_j\frac{1-\sigma_js_{i_j}}2=0.
\]
Together with $s_i^2=1$, these equations define a finite reduced scheme whose
points are precisely the executable directions. Two-path conditions have the
form $s_is_j=\kappa_{ij}$ and give a specified class in
$H^1(\Lambda;\mathbb F_2)$, where $\Lambda$ records the two-path constraints.
Longer circuits contribute their full equations. Root cuts alter the path
cover and hence its direction scheme; this is the repair operation measured
by $\tau_C$.

\subsection{The proof and the construction}
The proof first replaces an arbitrary terminal protocol by one of finitely
many modes without increasing its seed cost or its demands on the exterior.
A rooted polygon then exposes only four incidence ports, the incoming and
outgoing root occupations, and the causal reachability relation among those
ports. There are 56 such states. Gluing two child polygons to a triangle is
exact on these data: every compatible state pair lifts, and every parent
protocol restricts to a compatible pair.

To compare the unrestricted recurrence with the recurrence imposing core
capacities, we retain their two values on the same partial graph. A set
of 3,912 normalized data contains the primitive cases and is closed under
triangle composition. The two closing values agree for every datum and
every final terminal type. Induction on the number of core faces proves
the normal form, and attaining choices recover an optimal protocol.
The finite identities are specified in Appendix~A.

The 56-state recurrence constructs a minimum seed set and all predecessor
arrows in $O(n)$ time and $O(n)$ storage words from indexed triangle
adjacency. A topological order of the resulting incidence graph gives a
full chronology (Proposition~\ref{prop:constructive-algorithm}). Formula
\eqref{eq:intro-main} describes the capacity selections and root cuts
which attain this minimum.

The same boundary calculus yields exact rooted-block replacements. Responses
to fixed exterior protocols have 568 normalized types; after the exterior is
allowed to reoptimize, there are 284 additive context types, separated by
seventeen actual test polygons. Terminal and corridor reductions, including
\[
 Z(Q\cup_{ab}N)=Z(Q)+1,
\]
lead to the infinite families in Section~7. These include all-fan configurations
with a parity defect, shared-key configurations requiring two core edges,
and common-edge assemblies whose forcing number exceeds the sum of the two
local values by two.

\subsection{Relation with earlier work}
Zero forcing arose in the minimum-rank problem \cite{AIM}. Constrained
matchings and structural controllability give the viewpoint of \cite{TD}.
The forcing-arc and chain-twist theorem of \cite{CJMZ}, and parallel increasing
path covers \cite{PIP}, describe the causal conditions on a prescribed
path family. We use these criteria while restricting the optimizing edge
configurations by the polygon structure.

For triangulations, Hern\'andez, Ranilla, and Ranilla-Cortina \cite{HRR}
study forcing bounds and the construction of minimum sets. Domat and Kuenzel
\cite{DK} treat triangle paths, separator triangles, and fan leaves; Ison,
Kempton, and Kenter \cite{IKKpre} study bounds for biconnected
outerplanar graphs. Treewidth-based algorithms give another approach to exact
computation; see the introduction of~\cite{GHN} and the references cited there. The normal form proved here isolates the capacity
choices and direction repairs which determine the optimum within this class.

The support geometry and open equation complexes for general graphs are
studied in \cite{HLgeneral}. The related formulation by supported
residuals, positive clocks, and finite root states appears in
\cite{HLcomparison}. Here the polygon decomposition gives the core
capacities, terminal replacements, and constructive recurrence.

Section~2 establishes terminal normal forms and the core. Section~3 develops
causal states and exact response composition. Section~4 proves the capacity
comparison, and Section~5 derives \eqref{eq:intro-main}. Rooted replacements
and exact families occupy Sections~6 and~7. Algorithms and finite
certificates are collected in Appendix~A.

\section{Forcing arrows and terminal strips}\label{sec:strips}

\subsection{The incidence digraph}
A directed edge set $M$ is admissible as an arrow set if it contains no opposite pair and satisfies
\[
 \ind_M(v)\le1,\qquad \out_M(v)\le1\qquad(v\in V).
\]
These are the source and target conditions of a matching between row and column copies of the vertices. In the original graph, a vertex may occur once as a source and once as a target.

For each $v\in V$, introduce two incidence ports $p_v,q_v$. The port $p_v$
records the event that $v$ becomes black. The port $q_v$ records the use of
the row at $v$ for its selected outgoing force; if there is no such force,
$q_v$ is a terminal auxiliary event. Define $D_G(M)$ by
\begin{equation}\label{eq:incidence}
 p_v\longrightarrow q_v,\qquad
 \begin{cases}
 q_u\longrightarrow p_v,&u\to v\in M,\\
 p_v\longrightarrow q_u,&u\to v\notin M,
 \end{cases}\qquad u\sim v.
\end{equation}
The initial set prescribed by $M$ is
\[
 S(M)=V\setminus\{v:\exists u,\ u\to v\in M\}.
\]
We call $M$ executable if precisely its arrows can be performed from $S(M)$ in some chronological order. The predecessor arrows of the process are fixed by $M$.

\begin{lemma}\label{lem:incidence}
An arrow set $M$ is executable if and only if $D_G(M)$ is acyclic. In that case its components are directed induced paths, including isolated vertices, and $|S(M)|=|V|-|M|$.
\end{lemma}
\begin{proof}
For a chosen force $u\to v$, every vertex in $N[u]\setminus\{v\}$ must be black before $v$. Put $q_u$ after all these events and before $p_v$. An unused row $q_u$ is a sink and may be placed last. A chronological forcing process therefore supplies a strict ordering of every incidence arc.

Conversely take a topological order of $D_G(M)$, make $S(M)$ initially black, and process the other column vertices in that order. If $u\to v\in M$, then
\[
 p_w\longrightarrow q_u\longrightarrow p_v
       \qquad(w\in N[u]\setminus\{v\}).
\]
Thus $u$ and all its neighbours other than $v$ are already black, so the prescribed force is legal. All targets are distinct, proving the cardinality formula.

The in/out bounds make the selected graph a union of paths and possible cycles; strict times exclude cycles. If a chain $v_0\to\cdots\to v_k$ had a chord $v_iv_j$ with $j\ge i+2$, then $v_{i+1},v_j$ would both be white at the force $v_i\to v_{i+1}$. Thus each path is induced.
\end{proof}

The preceding criterion is an incidence form of the forcing-arc description in~\cite{CJMZ}. It gives
\begin{equation}\label{eq:arrow-opt}
 Z(G)=|V|-\max\{|M|:D_G(M)\text{ is acyclic}\},\qquad P(G)\le Z(G).
\end{equation}

\begin{lemma}\label{lem:deletion}
Deleting any selected arrow preserves executability. Reversing all selected arrows also preserves executability.
\end{lemma}
\begin{proof}
On deleting $u\to v$, the arc $q_u\to p_v$ is replaced by $p_v\to q_u$. Because $u$ has no remaining selected target and $v$ no remaining selected predecessor, $q_u$ becomes a sink and $p_v$ a source. The new arc creates no directed cycle. For reversal, transpose the whole incidence digraph and interchange $p_x,q_x$ for every $x$. Formula~\eqref{eq:incidence} becomes the incidence digraph for the reversed arrows.
\end{proof}

\subsection{Local tests in a polygon triangulation}
We regard a maximal biconnected outerplanar graph with at least three vertices as a convex polygon triangulation. A fixed outerplane embedding justifies this model: biconnectivity makes the outer boundary a simple cycle through all vertices, and maximality makes every bounded face a triangle. The auxiliary single-edge graph will also be used as a rooted block.

\begin{lemma}\label{lem:linear-forest}
For a polygon triangulation, $L\subseteq E$ is an induced linear forest spanning all vertices if and only if
\begin{equation}\label{eq:local-forest}
 \deg_L(v)\le2\quad(v\in V),\qquad
 |L\cap E(\Delta)|\le1\quad(\Delta\text{ a bounded triangular face}).
\end{equation}
Consequently $P(G)=|V|-\max|L|$ subject to \eqref{eq:local-forest}.
\end{lemma}
\begin{proof}
If $L$ is a union of induced paths, its vertex degrees are at most two. Two
selected sides of a triangle would lie consecutively in one component, and
the third side would be a chord of that path. Thus the two local conditions
are necessary.

Assume now that they hold. Since the selected degrees are at most two, any
component which is not a path must be a cycle. Suppose there is a selected
cycle with $k$ edges. In a convex polygon, no other polygon vertex lies
strictly inside the convex hull of those cycle vertices. The region bounded
by the cycle is therefore a triangulated $k$-gon and contains exactly $k-2$
triangular faces. Each of its $k$ boundary edges belongs to one face on its
inside. Two selected boundary edges cannot belong to the same face, by the
face condition. This would inject $k$ edges into $k-2$ faces, a contradiction.
Hence every selected component is a path or a singleton.

Suppose a selected path has a chord. Take the subpath between the chord's
ends; it has $k\ge2$ selected edges. Together with the unselected chord it
bounds a $(k+1)$-gon, again with no extra interior vertices. Its triangulation
has $k-1$ faces. Assigning each of the $k$ selected boundary edges to its
inward face gives the same forbidden injection. Thus the path is induced.
The whole selected graph is a spanning forest, whose component number is
$|V|-|L|$. Minimizing that number proves the final formula.
\end{proof}

The augmented weak dual has one vertex for each bounded triangular face and an additional leaf for each outer edge. Its edges correspond to the original graph edges: internal diagonals join adjacent faces and outer edges join their face to the additional leaf. It is a tree. For every polygon vertex $v$, its incident edges correspond to a path $\Pi_v$ in this tree, between the two exterior leaves adjacent to $v$.

\begin{corollary}\label{cor:dual-matching}
If $\widehat T_G$ is this augmented weak dual, then
\[
 P(G)=|V|-\max\{|M|:M\text{ is a matching of }\widehat T_G,
                     \ |M\cap E(\Pi_v)|\le2\text{ for all }v\}.
\]
\end{corollary}
\begin{proof}
A dual matching is precisely the one-selected-edge condition at each triangular face. The additional path inequality is the original vertex-degree condition in~\eqref{eq:local-forest}. The edge correspondence preserves cardinality.
\end{proof}
The paths $\Pi_v$, equivalently the corner identifications of the triangles, are retained as part of the augmented weak-dual data.

\subsection{Boundary comparison}
For an executable arrow set $M$ on a block $H$ rooted at $ab$, put
\[
 i_v=\ind_M(v),\qquad o_v=\out_M(v)\quad(v=a,b),\qquad
 \occ_H(M)=(i_a,i_b,o_a,o_b).
\]
Let $B_{ab}=\{p_a,q_a,p_b,q_b\}$, and let $R_H(M)$ be the positive-length
reachability relation of $D_H(M)$ restricted to $B_{ab}$. We write
\[
 \sigma_H(M)=(\occ_H(M),R_H(M))
\]
for its boundary state. The port order is $(p_a,q_a,p_b,q_b)$, whereas
occupation coordinates have the order $(i_a,i_b,o_a,o_b)$. Let
$\eta(M)\in\{0,+,-\}$ record respectively an unselected root, $a\to b$,
or $b\to a$. Two blocks meeting just in $ab$ have disjoint interiors.
The common selected root arrow is counted once in their union.

\begin{lemma}\label{lem:gluing}
Let $M_1,M_2$ be executable arrow sets on blocks meeting exactly in the
included root edge $ab$, with relations $R_1,R_2$. They glue to an executable
arrow set if and only if
\[
 \begin{gathered}
 \eta(M_1)=\eta(M_2),\\
 \occ(M_1)+\occ(M_2)-\occ(M_1\cap M_2)\le(1,1,1,1),\\
 \TC(R_1\cup R_2)\cap\{(z,z):z\in B_{ab}\}=\varnothing.
 \end{gathered}
\]
\end{lemma}
\begin{proof}
Let $M_1,M_2$ be the two local arrow sets. They define a single arrow set on the union precisely when they make the same choice on the common edge. Once this is satisfied, at a root vertex $v$ the global counts are
\[
 \ind_M(v)=\ind_{M_1}(v)+\ind_{M_2}(v)-\ind_{M_1\cap M_2}(v),
\]
and the analogous formula holds for outgoing arrows. The only possible common selected arrow is the root arrow. Thus the stated occupation inequalities are necessary and sufficient for the source and target constraints. Interior counts do not change.

Under these constraints, the union incidence graph is the union of the two local incidence graphs, with common arcs identified. Suppose it has a directed cycle. It cannot be contained in one block, because each local arrow set is executable. Break the cycle whenever it meets one of $p_a,q_a,p_b,q_b$. Each intervening segment lies in one block and is a positive-length path in its incidence graph. Replacing every such segment by its recorded root relation gives a directed closed walk in $R_1\cup R_2$, hence a diagonal pair in $\TC(R_1\cup R_2)$.

Conversely, a diagonal pair lifts to a nonempty directed closed walk in the
union incidence graph and hence to a directed cycle.
Lemma~\ref{lem:incidence} completes the equivalence.
\end{proof}

When the root is unselected, its mandatory relation is
\begin{equation}\label{eq:R0}
 R_0=\{p_a\to q_a,p_a\to q_b,p_b\to q_a,p_b\to q_b\}.
\end{equation}
By Lemma~\ref{lem:gluing}, a root-free arrow set can be replaced in any
exterior by one with no larger occupation, a contained root relation,
and at least as many arrows.

\begin{lemma}\label{lem:root-delete}
Suppose $a\to b$ is selected in an executable root block. Deleting it leaves root relation exactly $R_0$ and occupation $(i_a,0,0,o_b)$.
\end{lemma}
\begin{proof}
The selected root and the two diagonal incidences give the old chain
\[
 p_a\longrightarrow q_a\longrightarrow p_b\longrightarrow q_b.
\]
Acyclicity excludes every old relation in the opposite direction. Delete
$q_a\to p_b$ and insert $p_b\to q_a$. Since $a$ has no other selected
outgoing graph edge, $q_a$ is now a sink; since $b$ has no other selected
incoming graph edge, $p_b$ is now a source. In particular there is no new
path ending at $p_b$ or beginning at $q_a$.

Every path using the new arc starts at $p_b$ and stops at $q_a$: it cannot
have a nontrivial segment before that source or after that sink. Every other
new root relation was already an old relation. Among relations compatible
with the displayed old order, the only ones not in $R_0$ were $p_a\to p_b$
and $q_a\to q_b$, both now impossible by the source and sink conditions.
All four relations in $R_0$ are actual unmatched root or diagonal
incidences. Hence the new relation is exactly $R_0$.
The removed arrow was the sole outgoing arrow at $a$ and the sole incoming
arrow at $b$; the other two occupations are unchanged. This gives
$(i_a,0,0,o_b)$.
\end{proof}

\subsection{Terminal strips}
The terminal branches of the weak dual are paths of triangles. We describe
their attachments from the root edge in order to compute their forcing
modes.

\begin{definition}\label{def:strip}
A rooted terminal strip starts with an edge $ab$. At each step a new vertex
is joined to the endpoints of the current edge; if the construction
continues, the next current edge consists of that new vertex and one old
endpoint. A strip with $m$ added vertices has $m$ triangular faces. For
$m\ge2$, order its roots so that $b$ has the unique interior neighbour
$u_1$. The strip is a fan if $a$ is adjacent to every interior vertex;
in this case $a$ is its key and $b$ its nonkey. A nonfan strip is called a
serpent. We retain the ordered root pair also for serpents. A strip with
one triangular face is an ear.
\end{definition}
Write the added vertices as $u_1,\ldots,u_m$ and the fronts as $e_0=ab,e_1,\ldots,e_{m-1}$. For $i<m$, let $r_i$ be the endpoint of $e_{i-1}$ omitted from $e_i$. At stage $i$, the vertices of $e_{i-1}$ and all earlier added vertices are black. An omitted vertex never returns to a later front, so
\[
 N(r_i)\cap\{u_i,\ldots,u_m\}=\{u_i\}.
\]
Consequently $r_i\to u_i$ is legal. At the last stage both ends of $e_{m-1}$ are black and $u_m$ is the only remaining white vertex of the strip, so either end can force it. No selected source has been used before: a used source retired at its step, whereas the current endpoints have not retired. This constructs exactly $m$ arrows from the two roots. By Lemma~\ref{lem:incidence}, they form two disjoint rooted induced paths. For a fan, choose $u_{m-1}\to u_m$ at the final step. Its fixed full forest
is then the path $b,u_1,\ldots,u_m$, together with the singleton root $a$.
For a serpent, fix either legal choice at the final step and retain the
resulting two root-started paths throughout the paper. An ear is assigned
its demanded root when the ear allocation is chosen.

\begin{lemma}\label{lem:serpent-geometry}
The internal graph of a serpent contains a triangle. Its only degree-two interior vertex is the final vertex, and this vertex is adjacent to neither root. In a root-free induced path cover of a serpent with exactly two paths, both roots have selected degree one. If that cover is executable, both root arrows are directed into the roots, or both out of the roots.
\end{lemma}
\begin{proof}
Let $e_i$ be the fronts just defined. A vertex which leaves a front is never used in a later attachment. Since the strip is not a fan, both original roots eventually leave the front. If the last root leaves in passing to $e_j$, then $e_j$ has two interior endpoints and the next face adds a third interior vertex. Thus the internal graph contains a triangle. Every $u_i$ with $i<m$ gains a neighbour $u_{i+1}$ besides its two older neighbours. Only $u_m$ has degree two. Both original roots have already left its front, so neither is adjacent to $u_m$.

Now let a root-free induced path cover consist of two paths. The roots lie on different paths: if they lay on one, the unselected edge $ab$ would be its chord. Neither root can form a singleton path, for then the other induced path would contain the whole internal triangle. The root $b$ has only one available interior neighbour, $u_1$, so its selected degree is one. Call this path $Q$ and the path containing $a$ by $P$.

The interior neighbours of $a$ are an initial sequence $u_1,\ldots,u_h$ of the added vertices. Consecutive members of this sequence are adjacent: while $a$ remains in the front, the next front is $\{a,u_i\}$. Suppose that $a$ had two selected neighbours $u_i,u_j$, where $i<j$. Since $u_1\in Q$, we have $i\ge2$. Since a triangle cannot contribute two selected edges, $j\ge i+2$. Every other neighbour of $a$ must lie in $Q$; otherwise its edge to $a$ would be an unselected chord of $P$. In particular $u_{i+1}\in Q$.

The diagonal $au_i$ is the front separating the already constructed faces from the later faces. Its two sides meet only in its endpoints. One side contains $b,u_1$, and the other contains $u_{i+1}$. Both separator endpoints belong to $P$, so a path $Q$ disjoint from $P$ cannot connect $u_1$ to $u_{i+1}$. This contradiction proves $\deg_P(a)=1$.

Finally suppose the two root incidences have opposite directions. Because each root has selected degree one and lies in a different path, the prescribed initial set then consists of one root and one interior vertex $x$. The black root sees both the other, white root and its own white interior target. It cannot force. If $x\ne u_m$, then $\deg_H(x)\ge3$ and at most one neighbour, the black root, is black; it too has at least two white neighbours. If $x=u_m$, it is adjacent to neither root and both of its neighbours are white. There is therefore no first force from the prescribed initial set. Only the two equal root directions remain possible.
\end{proof}

A symbol $a\Iroot$ means that an interior arrow enters $a$; $a\Oroot$ means
that $a$ sends an arrow into the interior. For a strip with $m$ interior
vertices, write
\[
 w=m-|M_H|=|S(M_H)|-2.
\]
Thus $w$ is the cost normalized at the two roots. The table records this
cost, the root occupation, and the relation in addition to $R_0$:
\begin{center}\small
\begin{tabular}{llllr}
\toprule
Strip&Mode&$(i_a,i_b,o_a,o_b)$&Extra relation&$w$\\\midrule
Fan&empty&$0000$&none&1\\
&$b\Iroot$&$0100$&$p_a\to p_b$&0\\
&$b\Oroot$&$0001$&$q_b\to q_a$&0\\\midrule
Serpent&empty&$0000$&none&2\\
&$a\Iroot$&$1000$&none&1\\
&$a\Oroot$&$0010$&none&1\\
&$b\Iroot$&$0100$&$p_a\to p_b$&1\\
&$b\Oroot$&$0001$&$q_b\to q_a$&1\\
&$\Iroot\Iroot$&$1100$&$p_a\to p_b$&0\\
&$\Oroot\Oroot$&$0011$&$q_b\to q_a$&0\\\midrule
Ear&empty&$0000$&none&1\\
&$a\Iroot$&$1000$&$p_b\to p_a$&0\\
&$a\Oroot$&$0010$&$q_a\to q_b$&0\\
&$b\Iroot$&$0100$&$p_a\to p_b$&0\\
&$b\Oroot$&$0001$&$q_b\to q_a$&0\\\bottomrule
\end{tabular}
\end{center}

\begin{lemma}\label{lem:modes}
Every table entry is realized on the original strip. Every executable root-free arrow set is dominated by an entry: the entry has no greater cost or occupation, and its root relation is contained in the old relation.
\end{lemma}
\begin{proof}
For a fan, the internal graph is the path $u_1-\cdots-u_m$, and $b$ is adjacent internally only to $u_1$. Orienting the internal path by itself uses $m-1$ arrows and leaves both roots unused, giving the empty mode. Attach $b$ and orient the path $b,u_1,\ldots,u_m$ in either direction. These two realizations use $m$ arrows. For $u_1\to b$, the adjacency $a\sim u_1$ forces
\[
 p_a\to q_{u_1}\to p_b.
\]
For $b\to u_1$, the same adjacency gives
\[
 q_b\to p_{u_1}\to q_a.
\]
All other proposed extra relations would have to enter an unmatched root column or leave an unused root row. Such columns are sources and such rows are sinks. Thus the displayed pairs give the full extra relations.

For a serpent, use the fixed two-root construction. Its first arrow is $b\to u_1$, so the full outward mode has the relation $q_b\to p_{u_1}\to q_a$. Both root columns are sources; a relation $q_a\to q_b$ would make a cycle. No root row can reach either root column. These observations give exactly $R_0\cup\{q_b\to q_a\}$. Reversing all arrows, transposing the incidence graph and interchanging $p,q$, gives the inward mode with relation $R_0\cup\{p_a\to p_b\}$.

The remaining modes are obtained by deleting the arrow at each root that is to become inactive. If the $a$ arrow is removed, the relation forced through $u_1$ remains. If the $b$ arrow is removed, its row becomes a sink in the outward case, or its column becomes a source in the inward case, so the extra relation disappears. Removing both root incidences leaves only $R_0$. Each deletion adds one to $w=m-|M_H|$ and is safe by Lemma~\ref{lem:deletion}. This proves all seven realizations on the same fixed full forest. On one triangle the only root-free choices are the empty choice and its four single directed nonroot edges, giving the ear table.

For domination, the underlying selected forest has at least two components, since the strip contains a triangle. Hence $|M_H|\le m$ and $w\ge0$. In a fan, if $w\ge1$, the empty mode has no greater cost or occupation and relation $R_0$, which is present in every root-free state. If $w=0$, there are exactly two components. The root $b$ cannot be isolated, because the other path would contain $a,u_1,u_2$ and their triangle. Since $b$ has only one interior neighbour, it has exactly one selected incidence. Its direction is either $b\Iroot$ or $b\Oroot$, and the corresponding extra relation has already been shown to follow from that incidence. The table mode has cost zero and uses no other root slot.

In a serpent, $w\ge2$ is dominated by the empty mode. If $w=1$, some root incidence must be present: with neither root used, the $m$ interior vertices would have $m-1$ selected edges and hence form one induced path, contradicting their internal triangle. If $b$ is occupied, choose its direction from the table; its extra relation is forced through $u_1$. If only $a$ is occupied, choose one of its occupied directions; the table mode adds no extra relation. Both choices have cost one and no greater occupation. If $w=0$, Lemma~\ref{lem:serpent-geometry} proves that the two roots have degree one with equal directions. The corresponding full table mode has precisely the required occupation and its obligatory extra relation is already present. These alternatives exhaust all values of $w$.
\end{proof}

\begin{theorem}\label{thm:terminal-normal}
Let a terminal strip $H$ meet an arbitrary exterior $Q$ exactly along its root edge $ab$. Every executable global arrow set can be replaced, without decreasing its number of arrows, by one omitting $ab$ and using a table mode on $H$. All exterior arrows other than $ab$ remain unchanged.
\end{theorem}
\begin{proof}
If the root is unselected, Lemma~\ref{lem:modes} and Lemma~\ref{lem:gluing} give the assertion. Suppose that the selected root is $a\to b$. Let $M_H,M_Q$ be the restrictions of the global arrow set, with the common arrow recorded on each side. Delete it from $M_Q$. By Lemma~\ref{lem:root-delete}, the remaining exterior has relation $R_0$ and occupation
\[
 (\alpha,0,0,\beta),\qquad \alpha,\beta\in\{0,1\}.
\]
We construct a root-free mode $M_H'$ with $|M_H'|\ge|M_H|$ that fits these four slots.

An ear has at most one local arrow, and its $a\Oroot$ mode has one, in a newly free outgoing slot. A fan has at most $m$ arrows. If its nonkey is $b$, use $b\Iroot$; if its nonkey is $a$, use $a\Oroot$. Each has $m$ arrows and uses only a slot freed by the root deletion.

For a serpent, when $\alpha=0$ both incoming slots are free, so use $\Iroot\Iroot$. When $\beta=0$ both outgoing slots are free, so use $\Oroot\Oroot$. These modes again have $m$ arrows. The only uncovered occupation is $\alpha=\beta=1$. The exterior then occupies the incoming slot at $a$ and the outgoing slot at $b$, while the old root arrow occupies the other two. Hence no other local arrow can meet either root. The edge $ab$ is an entire path component of the old local cover. The interior contains a triangle, so its $m$ vertices have at most $m-2$ selected edges. Therefore
\[
 |M_H|\le1+(m-2)=m-1.
\]
Either $a\Oroot$ or $b\Iroot$ is available and has $m-1$ arrows.

In each case the new mode adds to $R_0$ at most one column-to-column or row-to-row relation. That union has no cycle, since $R_0$ goes only from columns to rows and contains no row-to-column path. Thus it is compatible with the unchanged exterior after deletion. Counting the shared arrow once gives
\[
 |M'|=(|M_Q|-1)+|M_H'|
 \ge |M_Q|+|M_H|-1=|M|.
\]
For the opposite root arrow, reverse every arrow, apply the proved case, and reverse back. Reversal preserves executability, and the list of modes is closed under this operation. No exterior arrow other than the root has been changed.
\end{proof}
The theorem may be applied successively to terminal strips with disjoint interiors. Starting from an optimum, a nondecrease in arrow number preserves optimality.

\subsection{An exact neutral replacement}
\begin{theorem}\label{thm:neutral}
Let $Q$ contain an edge $ab$. Add a new vertex $v$, a triangle $avb$, a fan along $av$ keyed at $a$, and a fan along $vb$ keyed at $b$. Each fan has at least two faces, and its interior is disjoint from $Q$ and from the other fan. Call the new rooted block $N$. Then
\[
 Z(Q\cup_{ab}N)=Z(Q)+1.
\]
Every prescribed exterior protocol has an extension with exactly one additional seed and no extension can have smaller additional cost.
\end{theorem}
\begin{proof}
Let the two fan interiors have $m_1,m_2$ vertices. Direct their fixed
nonkey paths into $v$ from the left and out of $v$ to the right. The
incoming mode on the fan rooted at $av$ adds $p_a\to p_v$; the outgoing
mode on the fan rooted at $bv$ adds $q_v\to q_b$. The unselected fan roots supply the column-to-row incidences in
\eqref{eq:R0}. Eliminating the ports of $v$
therefore adds the path
\[
 p_a\longrightarrow p_v\longrightarrow q_v\longrightarrow q_b,
\]
and no other extra relation between the ports of $a,b$. The modes use $v$
once as a target and once as a source, and use no slot at $a,b$.
Reversing both paths, transposing the incidence graph and interchanging
$p,q$ gives instead $p_b\to q_a$.

With the root unselected, both possible extra relations belong to $R_0$.
With root $a\to b$, the first belongs to the root chain; with root
$b\to a$, the second does. Hence every prescribed exterior protocol
extends. There are $m_1+m_2+1$ new vertices and $m_1+m_2$ new arrows,
giving one extra seed.

For a lower bound apply Theorem~\ref{thm:terminal-normal} to the two fan roots. This keeps all arrows of $Q$ unchanged. Each fan then has at most $m_i$ arrows, and there is no other new core arrow. The restricted exterior remains executable. Hence
\[
 |V(Q\cup N)|-|M|\ge |V(Q)|-|M_Q|+1\ge Z(Q)+1.
\]
The same count holds with the exterior protocol prescribed.
\end{proof}

\begin{lemma}\label{lem:ear-fan}
For a one-triangle ear $T$ on $ab$ and every exterior $Q$,
\[
 Z(Q\cup T)=\min\{Z(Q\cup F_a),Z(Q\cup F_b)\},
\]
where $F_a,F_b$ are two-face fans keyed at the indicated root.
\end{lemma}
\begin{proof}
Normalize the common root in an optimal protocol on $Q\cup T$ by
Theorem~\ref{thm:terminal-normal}. If its ear mode uses a root $v$, choose
the two-face fan whose nonkey is $v$. The corresponding fan mode has the
same occupation, root relation, and normalized cost $m-|M_H|$. If the ear
mode is empty, either fan has its empty mode, again with the same data.
Replacing the ear adds one interior vertex and one selected arrow. The
global seed number $|V|-|M|$ is therefore unchanged, and the gluing criterion
preserves executability. This proves
\[
 \min\{Z(Q\cup F_a),Z(Q\cup F_b)\}\le Z(Q\cup T).
\]
Conversely normalize any optimal protocol on either $Q\cup F_a$ or
$Q\cup F_b$. Every fan table mode compresses to an ear mode with the same
occupation, relation, and normalized cost. Both one vertex and one arrow
are removed, so the seed number is unchanged. The resulting protocol lies
on $Q\cup T$, proving the opposite inequality for each fan separately.
Together these inequalities give the claimed minimum.
\end{proof}

\subsection{The canonical core}\label{sec:core}
Let $T_G$ be the weak dual, with one node per bounded triangular face. It is a tree of maximum degree three. If it has no degree-three node, $G$ is a strip; the two-root construction and minimum degree two give $P(G)=Z(G)=2$.

Otherwise let $T_0$ be the smallest connected subtree containing every degree-three node, and let $C$ be the union of its triangular faces. Every component of $T_G\setminus T_0$ is a terminal path attached by one edge. It gives a terminal strip $H_i$ on a boundary edge of $C$.

\begin{lemma}\label{lem:core}
The graph $C$ is a polygon triangulation. Distinct removed strips have disjoint interiors and meet $C$ only in their roots. If $\ell$ is the number of weak-dual leaves and $d_2$ is the number of original degree-two dual nodes retained in $T_0$, then
\[
 \#\{H_i\}=\ell,\qquad |V(C)|=\ell+d_2.
\]
\end{lemma}
\begin{proof}
Root the finite dual subtree $T_0$ at one of its triangles and add its faces in parent-before-child order. A joining dual edge corresponds to a diagonal of the original polygon. This diagonal separates the original polygon into two closed polygonal regions meeting exactly in that edge. All faces already attached on the parent's side lie in one region, and the next face lies on the other side. Its third vertex is therefore new to the union. Induction shows that $C$ is a triangulated polygon, with one more vertex at every added face.

The faces incident with any original vertex form a path in $T_G$: consecutive faces in their circular order share precisely the incident diagonal separating them. A path and a connected subtree of a tree have connected intersection. If a vertex outside $C$ belonged to two different components of $T_G\setminus T_0$, the dual path of its incident faces would pass through $T_0$, placing it in $C$, a contradiction. Similarly, the separating diagonal at the attachment of a removed component shows that its only vertices in $C$ are the two root endpoints. Thus the strip interiors are disjoint and have exactly the claimed intersections.

No removed component contains a degree-three node, since all such nodes lie in $T_0$. Each component attaches to $T_0$ by one dual edge; attaching by two would create a cycle in $T_G$. It is consequently a path, ending in exactly one original leaf. Conversely $T_0$ contains no original leaf: a leaf cannot lie on a path between two degree-three nodes, and if there is only one such node then $T_0$ is that node. Hence the removed components correspond bijectively to all $\ell$ original leaves.

Let $b$ be the number of degree-three nodes and $d$ the total number of degree-two nodes in $T_G$. Its degree sum gives
\[
 \ell+2d+3b=2(\ell+d+b-1),\qquad b=\ell-2.
\]
The nodes of $T_0$ consist of those $b$ nodes and its $d_2$ retained degree-two nodes. Thus $C$ has $b+d_2$ triangular faces and
\[
 |V(C)|=b+d_2+2=\ell+d_2.
\]
\end{proof}

Let $\cR$ be the set of root edges of the nonempty strips. A core
boundary edge outside $\cR$ is an original outer edge; we call it a
bare side. Bare sides are eligible for core arrows. The core order is
$|V(C)|=\ell+d_2$, and equals $\ell$ when $d_2=0$.

\section{Causal boundary states and weighted composition}\label{sec:responses}

\subsection{The 56 protocol states}
A rooted block includes its ordered root edge. Its boundary state is
$\sigma_H(M)$ as defined in Section~\ref{sec:strips}. Relations in this
section are transitively closed and contain no diagonal pair.

Let $E$ be the single-edge block with an ordered root. Given rooted blocks
$H,K$, place disjoint copies on the ordered roots $(a,c)$ and $(c,b)$,
identify their common vertex $c$, and add the edge $ab$. The resulting
block, rooted at $(a,b)$, is denoted by $T(H,K)$. Set
\[
 L(H)=T(H,E),\qquad R(H)=T(E,H).
\]
As a terminal type, $T$ denotes the one-triangle block $T(E,E)$.

\begin{lemma}\label{lem:56}
For a rooted polygon triangulation, the root relation has one of the following seven forms:
\[
\begin{array}{c|l|l}
0&R_0&\text{no additional occupation condition}\\
C_+&R_0\cup\{p_a\to p_b\}&i_b=1\\
C_-&R_0\cup\{p_b\to p_a\}&i_a=1\\
R_+&R_0\cup\{q_a\to q_b\}&o_a=1\\
R_-&R_0\cup\{q_b\to q_a\}&o_b=1\\
E_+&p_a<q_a<p_b<q_b&o_a=i_b=1\\
E_-&p_b<q_b<p_a<q_a&o_b=i_a=1.
\end{array}
\]
Together with binary occupations these give a closed set $\Sigma$ of 56
states. Every executable arrow set on a rooted polygon triangulation has
a state in $\Sigma$. Conversely, each $s\in\Sigma$ is realized by an
executable arrow set on some rooted polygon triangulation; the realizing
polygon may depend on $s$.
\end{lemma}
\begin{proof}
For the auxiliary single edge, the three choices give $R_0,E_+,E_-$. Assume the assertion for the two children of a root triangle $acb$.

If the root arrow is $a\to b$, then the root itself supplies the chain
\[
 p_a\to q_a\to p_b\to q_b.
\]
Its transitive closure already contains every pair in this total order. Any additional pair contrary to the order would make a cycle. Therefore the entire relation is $E_+$. A root arrow $b\to a$ gives $E_-$.

Suppose the root is unselected. Then $R_0$ is present. A row-to-column relation is impossible, since $R_0$ contains its reverse. Beyond $R_0$, the only possible pairs are
\[
 p_a\to p_b,\quad p_b\to p_a,\quad
 q_a\to q_b,\quad q_b\to q_a.
\]
Opposite pairs cannot coexist. We prove that a column pair and a row pair cannot coexist either.

Both child roots cannot be selected. If their orientations have the same source or target at $c$, the corresponding occupation exceeds one. The other alternatives are chains $a\to c\to b$ and $b\to c\to a$. The first gives
\[
 q_a\to p_c\to q_c\to p_b\to q_a,
\]
where the last incidence uses the unselected edge $ab$; the second gives the reversed analogue. Both contradict acyclicity.

If exactly one child root is selected, its relation is a total order on its four ports. Suppose first it is $a\to c$. The left child gives $p_a<q_a<p_c<q_c$, and the right child's unselected root gives $p_c\to q_b$. Hence $q_a\to q_b$ is forced. A reverse row relation would immediately make a cycle. A column relation $p_b\to p_a$ cannot occur: a path entering the left child through $p_c$ or $q_c$ cannot move backwards to $p_a$ in its total order. For $p_a\to p_b$, a path must cross the interface through $p_c$ or $q_c$. The right child has no row-to-column passage, so it must provide $p_c\to p_b$. Then $q_a\to p_c\to p_b\to q_a$ is a cycle. Thus only $R_+$ occurs.

Interchanging $a,b$, or transposing the whole incidence graph and exchanging the row and column names, preserves all occupation and acyclicity conditions. These two operations give the other three cases without changing the argument:
\[
\begin{array}{c|c|c}
 \text{selected child root}&\text{forced extra relation}&\text{type}\\\hline
 a\to c&q_a\to q_b&R_+\\
 c\to a&p_b\to p_a&C_-\\
 c\to b&p_a\to p_b&C_+\\
 b\to c&q_b\to q_a&R_-.
\end{array}
\]

If neither child root is selected, neither child has a row-to-column passage. Nor does the unselected new root provide one. Consequently, any path from a column to a column can use only column-to-column passages. To go from the $a$ side to the $b$ side it must pass through $p_c$, using a column passage in each child. Similarly a row-to-row passage across the two children must use $q_c$ and a row passage in each. By induction neither child can provide both kinds. Thus the output has at most one additional pair and has one of the first five listed forms.

If $p_a\to p_b$ occurs, the last incidence entering $p_b$ must be a selected incoming arrow at $b$, so $i_b=1$. If $q_a\to q_b$ occurs, the first incidence leaving $q_a$ is a selected outgoing arrow at $a$, so $o_a=1$. The remaining requirements follow by the same two symmetries. The root chains require their selected source and target slots. Counting the admissible binary occupations gives
\[
 2^4+4\cdot2^3+2\cdot2^2=56.
\]
The preceding argument proves containment and closure. For the converse,
the finite certificate starts from the three states of $E$ and gives a
recursive triangle witness for each member of $\Sigma$. Evaluating these
witnesses with the six-port rule realizes all 56 states. The three edge
states are also realized on a triangle by leaving its third vertex initial.
These witnesses provide the exterior protocols used in
Proposition~\ref{prop:congruence}.
\end{proof}

Let $\eta(s)\in\{0,+,-\}$ be its root choice, determined by whether the relation is $E_+,E_-$, and put $\eps(s)=\one_{\eta(s)\ne0}$. For states $s_L,s_R$ at $(a,c),(c,b)$ and a new root choice $\eta$, define $\Phi(s_L,s_R,\eta)$ as follows. Add occupations at $a,c,b$, including the new root arrow; reject any value greater than one. Put both child relations on the six ports, add the root incidences of~\eqref{eq:incidence}, and take transitive closure. Reject a diagonal entry; otherwise restrict to the four root ports. Lemma~\ref{lem:56} shows that every accepted output belongs to $\Sigma$. Its full transition relation has 1,410 triples.

\subsection{Minimum-cost responses}
Products of matrices and vectors in this section use the min-plus semiring:
\[
 (A\otimes x)_t=\min_s(A_{t,s}+x_s),\qquad
 (A\otimes B)_{t,u}=\min_s(A_{t,s}+B_{s,u}).
\]
The value $+\infty$ denotes the absence of a finite candidate.
Define the raw vector
\[
 r_H(s)=|V(H)|-2-\max\{|M|:M\text{ executable on }H,\ \sigma_H(M)=s\},
\]
using $+\infty$ for an unrealized state. If $s,t$ are compatible in Lemma~\ref{lem:gluing}, write $s\sim t$, and set
\begin{equation}\label{eq:pairing-C}
 \mathsf C_{t,s}=\begin{cases}\eps(s),&s\sim t,\\+\infty,&s\not\sim t.\end{cases}
 \qquad
 \beta_H(t)=\min_s(\mathsf C_{t,s}+r_H(s)).
\end{equation}
Compatibility subtracts the one common root arrow from occupations. The same subtraction contributes the $+\eps(s)$ to the cost. There are 387 finite entries of $\mathsf C$.

\begin{proposition}\label{prop:response}
For rooted polygon blocks meeting only in their common included root,
\begin{equation}\label{eq:response-gluing}
 Z(H\cup_{ab}Q)=2+\min_t(r_Q(t)+\beta_H(t)).
\end{equation}
With an exterior protocol of state $t$ prescribed, $\beta_H(t)$ is exactly the part of the optimum depending on $H$.
\end{proposition}
\begin{proof}
For a fixed compatible pair of arrow sets $M_H,M_Q$, let $s,t$ be their states. The common root is selected on both sides or on neither. Hence
\[
 \begin{aligned}
 n_{H\cup Q}&=n_H+n_Q-2,\\
 |M_H\cup M_Q|&=|M_H|+|M_Q|-\eps(s).
 \end{aligned}
\]
Subtracting yields the cost
\[
 2+(n_H-2-|M_H|)+(n_Q-2-|M_Q|)+\eps(s).
\]
The compatibility criterion proves that each such pair is executable and every global executable set restricts to such a pair. For fixed $s,t$, independent minimizing choices in the two interiors can therefore be made simultaneously, giving $2+r_H(s)+r_Q(t)+\eps(s)$. Taking the minimum first over $s\sim t$, and then over $t$, proves the formula. For a prescribed $M_Q$, replace $r_Q(t)$ by its actual fixed cost $n_Q-2-|M_Q|$ in this calculation. The only remaining part to optimize is exactly $\beta_H(t)$.
\end{proof}
All entries of $\beta_H$ are finite: choose only the prescribed root arrow in $H$, or no arrow when the root is unselected, and leave every other local vertex initial. This has the minimal compatible boundary occupation and relation for that root choice.

For two raw vectors define
\begin{equation}\label{eq:triangle-F}
 \mathcal F(x,y)(s)=
 \min_{\Phi(u,v,\eta)=s}
       \bigl(x(u)+y(v)+1-\one_{\eta\ne0}\bigr).
\end{equation}
The root triangle composition $T(H,K)$ satisfies
\[
 r_{T(H,K)}=\mathcal F(r_H,r_K),
\]
because the two child blocks share one vertex and have disjoint edge sets, while the parent adds its root edge.

Put
\[
 \lambda_H=\min_t\beta_H(t),\qquad b_H=\beta_H-\lambda_H.
\]

\begin{proposition}\label{prop:congruence}
If $\beta_H=\beta_{H'}+d$, then for every rooted polygon block $K$,
\[
 \beta_{T(H,K)}=\beta_{T(H',K)}+d,
 \qquad
 \beta_{T(K,H)}=\beta_{T(K,H')}+d.
\]
\end{proposition}
\begin{proof}
Fix a parent exterior state $t$ and an actual executable exterior protocol $Q_t$ realizing it. Such witnesses for all states were included in Lemma~\ref{lem:56}. In the graph obtained by gluing $Q_t$ to $T(H,K)$, hold the arrows of $Q_t$ fixed and optimize all other arrows.

Remove the interior of $H$. The remaining graph $E$ consists of $K$, the parent triangle and $Q_t$, rooted at the former root of $H$. It is again a polygon triangulation. Every candidate arrow set on $E$ which restricts to the prescribed arrows of $Q_t$ is either locally nonexecutable and contributes no global solution, or is executable and has some state $u$. For such a candidate, put
\[
 a_E=|V(E)|-2-|M_E|.
\]
By the prescribed version of Proposition~\ref{prop:response}, its best extension through $H$ has total cost
\[
 2+a_E+\beta_H(u).
\]
Replacing $H$ by $H'$ does not change the graph $E$, its candidate arrows, or $a_E$. It replaces this cost by
\[
 2+a_E+\beta_{H'}(u)=2+a_E+\beta_H(u)-d.
\]
Every candidate changes by the same constant. Thus the optimum with the parent exterior prescribed changes by $d$. Subtract its fixed exterior cost, as in Proposition~\ref{prop:response}, to obtain the first asserted identity at coordinate $t$. Since $t$ was arbitrary, the first vector identity follows. Grouping the
opposite child with the triangle and exterior proves the second.
\end{proof}

\begin{lemma}\label{lem:568}
There are 568 distinct vectors $b_H$. Each is realized by a rooted polygon. They have a complete weighted product
\[
 \mu:\{0,\ldots,567\}^2\longrightarrow\{0,\ldots,567\},\qquad
 \gamma:\{0,\ldots,567\}^2\longrightarrow\{0,1,2,3\},
\]
such that
\begin{equation}\label{eq:568product}
 \beta_{T(H,K)}=\lambda_H+\lambda_K+
             \gamma(i,j)+b_{\mu(i,j)}
\end{equation}
when $b_H=b_i,b_K=b_j$.
\end{lemma}
\begin{proof}
The certificate gives 568 distinct normalized vectors $b_i$, raw vectors
$x_i$, integers $\lambda_i$, and rooted triangle terms $R_i$. It verifies
\[
 r_{R_i}=\lambda_i+x_i,\qquad
 \mathsf C\otimes x_i=b_i,\qquad \min b_i=0.
\]
The first identity is evaluated recursively from the three edge choices.
Thus every listed vector is realized by a rooted polygon.

For every ordered pair $(i,j)$, form
\[
 y_{ij}=\mathcal F(x_i,x_j),\qquad
 g_{ij}=\min(\mathsf C\otimes y_{ij}).
\]
For every ordered pair, the finite calculation gives
\[
 \mathsf C\otimes y_{ij}=g_{ij}+b_{\mu(i,j)},\qquad
 g_{ij}=\gamma(i,j)\in\{0,1,2,3\}
\]
for all $568^2=322,624$ pairs. The states, transitions and pairing matrix
are reconstructed from the incidence definitions. The certificate and
verification procedure are specified in Appendix~\ref{app:finite-proof}.

We now induct on the number of faces. The edge block is a listed type with
scalar zero. Suppose $\beta_H=\lambda_H+b_i$ and
$\beta_K=\lambda_K+b_j$. Proposition~\ref{prop:congruence} allows each child
to be replaced by its representative while retaining the scalar
difference. The verified pair identity gives
\[
 \beta_{T(H,K)}=\lambda_H+\lambda_K+\gamma(i,j)+b_{\mu(i,j)}.
\]
Every rooted polygon has this recursive decomposition, so the list is
exhaustive. The realized representatives and their distinct normalized
vectors show that all 568 types occur.
\end{proof}

For the edge type, $\lambda=0$ and $b=0$. Define
\[
 \delta(i)=2+\min_{s\text{ an edge state}}(b_i(s)-\eps(s)).
\]
The certificate gives $\delta\in\{1,2\}$, with 38 and 530 types respectively. Induction in~\eqref{eq:568product}, followed by~\eqref{eq:response-gluing} against the edge block, proves
\begin{equation}\label{eq:568-eval}
 Z(G)=\delta(i_{\rm root})+
       \sum_{\Delta}\gamma(i_{\Delta,L},i_{\Delta,R}).
\end{equation}
The triangulation fixes the rooted triangle tree and the parenthesization in~\eqref{eq:568-eval}.

\section{Hard capacities on the core}\label{sec:capacities}

\subsection{The directed core problem}
Let $C$ be a triangulated polygon, with terminal strips on a subset $\cR$ of its boundary edges. The other sides are bare. Theorem~\ref{thm:terminal-normal} reduces every optimum to core arrows on $E(C)\setminus\cR$ and the listed terminal modes.

Write $M_C$ for the selected core arrows and
$J=\{\{u,v\}:u\to v\in M_C\}$ for their underlying edge set. In the
path-cover and phase formulas, $J$ always denotes an undirected core
selection. For modes $\mu_i$, let $w_i$ be their normalized costs. Root occupations must satisfy
\begin{equation}\label{eq:core-occupations}
 \ind_{M_C}(v)+\sum_i i_v^{(i)}\le1,\qquad
 \out_{M_C}(v)+\sum_i o_v^{(i)}\le1.
\end{equation}
Form $N_C(M_C,\mu)$ from $D_C(M_C)$ by adding the one extra root relation of each mode, when present. All original core adjacencies, including the unselected strip roots, remain in $D_C(M_C)$.

\begin{proposition}\label{prop:core-formula}
With $r=|V(C)|$,
\begin{equation}\label{eq:soft-core}
 Z(G)=r+\min_{M_C,\mu}\left(\sum_iw_i-|M_C|\right),
\end{equation}
where the minimum is subject to \eqref{eq:core-occupations}, at most one selected edge per core triangle, and acyclicity of $N_C(M_C,\mu)$.
\end{proposition}
\begin{proof}
Normalize any optimum by Theorem~\ref{thm:terminal-normal}. Every remaining core arrow and mode satisfies the stated conditions. Conversely realize every selected mode in its actual strip. Occupations make the union a directed matching. A cycle cannot lie entirely inside a strip; replace each excursion through a strip by the recorded root relation. This gives a cycle of $N_C$. Conversely its cycles lift to closed directed walks in the full incidence graph. Thus the stated conditions exactly give executability. There are
\[
 r+\sum_i m_i\text{ vertices},\qquad |M_C|+\sum_i(m_i-w_i)\text{ arrows}.
\]
Subtracting gives~\eqref{eq:soft-core} in both directions.
\end{proof}

For each serpent put $U_i=\{a_i,b_i\}$, and for each fan let $U_i$ consist of its nonkey root. For an ear choose one demanded root $\theta_i\in\{a_i,b_i\}$ and put $U_i=\{\theta_i\}$. An empty ear still retains this choice as part of its data. Let
\begin{equation}\label{eq:capacities}
 t_v=\#\{i:v\in U_i\},\qquad c_v=2-t_v\in\{0,1,2\}.
\end{equation}
Each vertex is incident with at most two boundary strips, so these capacities are nonnegative. On a fully fan/serpent-decorated core, $c_v$ is the number of adjacent fans keyed at $v$.

Each demanded terminal connection uses one degree at its root. Reserving
all these degrees gives $\deg_J(v)\le c_v$. In
\eqref{eq:soft-core}, an unused connection instead contributes one to the
terminal cost. We prove that the reserved-degree condition can be imposed
while preserving the optimum.

\begin{theorem}\label{thm:hard-full}
Suppose every boundary side of $C$ carries a terminal fan or serpent of at least two faces. There is an optimal protocol in Proposition~\ref{prop:core-formula} satisfying
\begin{equation}\label{eq:hard-capacity}
 \deg_J(v)\le c_v\qquad(v\in V(C)),
\end{equation}
where the degree ignores arrow directions.
\end{theorem}
We prove the theorem by composing partial cores with their root degrees
retained. The restricted and unrestricted values on each partial core
give a finite comparison which is preserved by triangle composition.

\subsection{Partial-core states}
We use four primitive terminal blocks, all rooted at $(a,b)$. The fan
$F_a$ has faces $(a,b,x)$ and $(a,x,y)$; $F_b$ is obtained by interchanging
$a,b$. The serpent $S_a$ has faces $(a,b,x)$, $(a,x,y)$ and $(x,y,z)$;
$S_b$ is its root reflection. Thus $b$ is the single-neighbour root of
$S_a$, whereas $a$ is the single-neighbour root of $S_b$. The reflection interchanges the two ordered root roles.

Root a partial core at $ab$. Every boundary side other than $ab$ carries its given fan or serpent. At either root only one of the two eventual boundary strips is inside the piece. Let $k_a,k_b\in\{0,1\}$ record whether that visible strip is a fan keyed at the indicated endpoint.

A restricted label is
\[
 \vartheta=(s,j_a,j_b),\qquad s\in\Sigma,\quad j_a,j_b\in\{0,1,2\}.
\]
The numbers $j_a,j_b$ count selected core edges at the roots, including
the root edge when selected. The state $s$ contains the total incoming
and outgoing occupations. At every nonroot core vertex the hard capacity
is imposed. The value of a realization is
\[
 |V(\text{decorated piece})|-2-|M|.
\]
For a leaf strip its root is unselected, $j_a=j_b=0$, and its values are precisely the costs of Lemma~\ref{lem:modes}. Its credits are $(1,0),(0,1),(0,0),(0,0)$ for $F_a,F_b,S_a,S_b$ respectively.

At a root triangle $acb$, combine left and right labels and select $\eta\in\{0,+,-\}$ on $ab$. With $e=\one_{\eta\ne0}$, impose
\begin{equation}\label{eq:forget-capacity}
 j_c^L+j_c^R\le k_c^L+k_c^R.
\end{equation}
This is the full capacity at $c$: the two visible strips are exactly its two outer-side strips. Set
\begin{equation}\label{eq:capacity-update}
 j_a=j_a^L+e,\qquad j_b=j_b^R+e,\qquad
 (k_a,k_b)=(k_a^L,k_b^R),
\end{equation}
and combine causal states by $\Phi$. The new value is
\begin{equation}\label{eq:capacity-cost}
 v_L+v_R+1-e.
\end{equation}
These updates are the vertex and arrow counts for one triangle.

\begin{lemma}\label{lem:dominance}
With credits fixed, a label $(s',j'_a,j'_b;v')$ safely replaces $(s,j_a,j_b;v)$ in every future capacity-constrained gluing if
\begin{equation}\label{eq:dominance}
 \begin{gathered}
 \eta(s')=\eta(s),\qquad j'_a\le j_a,\quad j'_b\le j_b,\quad v'\le v,\\
 \occ(s')\le\occ(s),\qquad R_{s'}\subseteq R_s.
 \end{gathered}
\end{equation}
First retain the least value for each label $(s,j_a,j_b)$. Among the
resulting weighted labels, $x'$ strictly dominates $x$ if $x'\ne x$ and
all comparisons in~\eqref{eq:dominance} hold. The frontier consists of the
labels with no strict dominator. Replacing either input by its frontier
before a triangle product gives the same output frontier as taking the
frontier after the product.
\end{lemma}
\begin{proof}
The unchanged root choice preserves the shared root arrow. Decreasing occupations and core degrees preserves all capacity inequalities. The new union relation is contained in the old one; transitive closure is monotone, so no cycle is introduced. Equations~\eqref{eq:capacity-update}--\eqref{eq:capacity-cost} decrease the output degrees and cost by the same comparison and leave output credits fixed. Thus a dominating child gives a dominating parent for every old allowed product. After taking the minimum at each label, mutual domination forces equal
states, core degrees and costs. Hence strict domination is acyclic on the
finite set of weighted labels. If a dominator is removed, its strict
domination chain therefore terminates at a retained label. Every discarded
parent candidate is dominated by one obtained from retained children, so
the two frontier operations agree.
\end{proof}

When the last boundary strip is closed, its credits $(l_a,l_b)$ complete the two root capacities. A label $\vartheta$ may be closed with its root-free mode $u$ precisely when the causal states are compatible, the partial root is unselected, and
\begin{equation}\label{eq:closing-capacity}
 j_a\le k_a+l_a,\qquad j_b\le k_b+l_b.
\end{equation}
The seed cost is $2+v_\vartheta+w_u$.

\subsection{The finite comparison and its graph-theoretic consequence}
Carry the unrestricted response algebra of Lemma~\ref{lem:568} alongside the restricted frontier. A joint datum is
\begin{equation}\label{eq:joint-record}
 (k_a,k_b;\ A;\ i;\ d).
\end{equation}
Regard a frontier as a cost function on the finite label set, with value
$+\infty$ outside its retained support. If $V_H$ is the restricted
frontier of a partial graph, put
\[
 h=\min_{\vartheta:V_H(\vartheta)<\infty}V_H(\vartheta),\qquad
 A(\vartheta)=V_H(\vartheta)-h
\]
on its finite support, and set $A(\vartheta)=+\infty$ elsewhere. Thus
$\min A=0$ and $V_H=h+A$. For the unrestricted response of the same partial graph, write
\[
 \beta_H=\lambda_H+b_i=h+d+b_i,
 \qquad d=\lambda_H-h.
\]
The frontier retains $j_a,j_b$ at the current roots and imposes the
capacities at all forgotten core vertices. The offset $d$ compares its
minimum with the unrestricted response minimum.
To multiply data, use~\eqref{eq:forget-capacity}--\eqref{eq:capacity-cost},
take the least cost at each label, pass to the frontier of
Lemma~\ref{lem:dominance}, and let $u$ be the new minimum to subtract. The unrestricted update is
\begin{equation}\label{eq:joint-product}
 i'=\mu(i,j),\qquad d'=d_L+d_R+\gamma(i,j)-u.
\end{equation}
Indeed, the new restricted minimum and unrestricted scalar are
\[
 h'=h_L+h_R+u,\qquad
 \lambda'=(h_L+d_L)+(h_R+d_R)+\gamma(i,j),
\]
so $d'=\lambda'-h'$ gives~\eqref{eq:joint-product}.

\begin{lemma}\label{lem:3912}
There is a set $\mathcal K$ of 3,912 distinct
data~\eqref{eq:joint-record} with the following properties.
\begin{enumerate}
\item The data of all four primitive decorations belong to $\mathcal K$.
\item Every datum is realized by a recursive partial core.
\item Every ordered product~\eqref{eq:joint-product} belongs to
      $\mathcal K$.
\item For each datum and each final primitive strip, the restricted
      closing minimum~\eqref{eq:closing-capacity} equals the
      unrestricted closing value.
\end{enumerate}
\end{lemma}
\begin{proof}
Let $\mathcal K$ be the listed set of 3,912 normalized data. The finite
verification establishes four assertions.

First, the four primitive data belong to $\mathcal K$. Their restricted
frontiers are obtained from the terminal mode table, and their
unrestricted entries are the responses of the same primitive graphs.

Second, every datum has a recursive witness. Evaluating that witness
verifies its credits, restricted frontier $h+A$, and unrestricted response
$h+d+b_i$ on one partial graph.

Third, for every ordered pair of data, the calculation enumerates every
pair of finite child labels and every root choice. It applies
\eqref{eq:forget-capacity}, the six-port rule,
\eqref{eq:capacity-update} and~\eqref{eq:capacity-cost}, and then takes the
minimum at each label and the frontier of the result. If $u$ is the minimum
subtracted, the resulting normalized datum, with
\[
 i'=\mu(i,j),\qquad
 d'=d_L+d_R+\gamma(i,j)-u,
\]
belongs to $\mathcal K$. This checks all
$3,912^2=15,303,744$ products.

Finally, close each datum with each of the four primitive strips. For a
closing strip $B$ of unrestricted type $j$ and scalar $\lambda_B$, the two
normalized expressions are
\[
 \begin{aligned}
 z_{\rm cap}&=2+\min_{\vartheta,v}\bigl(A(\vartheta)+w_v\bigr),\\
 z_{\rm free}&=2+d+\lambda_B+\min_s\bigl(b_i(s)+x_j(s)\bigr).
 \end{aligned}
\]
The first minimum includes exactly the causal compatibility, unselected
root and capacity conditions in~\eqref{eq:closing-capacity}. An empty
minimum is $+\infty$. Equality holds in all
$4\cdot3,912=15,648$ closings.

Appendix~\ref{app:finite-proof} specifies the certificate and the
integer operations for these four assertions. The calculation uses the
incidence states of Lemma~\ref{lem:56} and the response algebra of
Lemma~\ref{lem:568}.
\end{proof}

\begin{proof}[Proof of Theorem~\ref{thm:hard-full}]
We first remove the length of a strip from the comparison. Terminal normalization replaces any unrestricted optimum by one with each terminal root unselected and each strip in a table mode. The same mode exists on its primitive representative with the same value of $m-|M_H|$, the same root occupation and the same root relation. It also exists back on the original strip. Replacing all strips by their representatives therefore preserves the unrestricted optimum in both directions. In the restricted problem it preserves the optimum as well: core arrows and capacities have not changed. It is enough to compare the two problems for primitive decorations.

Choose a core outer side as the final root and reserve its terminal strip for closing. The remaining decorated core is a binary triangle term, with its other terminal strips at the leaves. The induction invariant for a partial term is the following: its datum represents the exact restricted frontier after domination, enforcing all forgotten capacities, and also represents its exact unrestricted response, with the common scalar $h$ restored as $h+A$ and $h+d+b_i$ respectively.

This invariant holds at a leaf by the explicit primitive modes. At a triangle, its shared vertex $c$ sees exactly two outer-side strips, one in each child. Hence $k_c^L+k_c^R$ is its complete final capacity, and~\eqref{eq:forget-capacity} imposes exactly the missing inequality before $c$ is forgotten. The root degrees and costs are precisely~\eqref{eq:capacity-update}--\eqref{eq:capacity-cost}. Every feasible pair of child realizations gives the parent realization, and every parent restricts to such a pair by the interface criterion. Thus their minima are exact. Lemma~\ref{lem:dominance} permits discarding dominated labels without affecting any future feasible optimum. The unrestricted induction follows independently from the response congruence. Lemma~\ref{lem:3912} supplies the next datum for every such product, and~\eqref{eq:joint-product} maintains the scalar invariant.

At the root, the reserved terminal strip supplies the missing credits at the last two core vertices. The closing rule imposes both remaining hard inequalities and all causal constraints. The terminal comparison in Lemma~\ref{lem:3912} identifies this restricted minimum with the unrestricted zero forcing number of the whole decorated graph.

Every finite entry was obtained by a minimum over actual choices. Store one attaining pair whenever a minimum is formed. Starting with an attaining root closure and following these choices to the leaves recovers an executable arrow set satisfying all hard capacities. Reinstall the original strip in the same mode at each leaf. The boundary relations and normalized costs are unchanged, so the resulting arrow set on the original graph is executable and has the same, optimal, number of seeds. This proves existence of the claimed optimum as well as equality of the two numerical minima.
\end{proof}

\subsection{Ears and bare sides}
\begin{theorem}\label{thm:hard-general}
For every core $C$ with terminal strips on some boundary sides, there is an ear allocation $\theta$ and an optimal forcing protocol which omits all terminal roots, uses the canonical terminal modes, and has core arrows satisfying
\[
 M_C\subseteq\overrightarrow{E(C)\setminus\cR},\qquad
 \deg_J(v)\le c_v=2-t_v.
\]
The underlying core edge set $J$ has at most one selected edge in each core triangle. Bare core sides remain allowed.
\end{theorem}
\begin{proof}
Choose an optimal protocol on $G$ and apply terminal normalization. For each
ear, let its demanded root be the one used by its mode; for an empty mode
choose either root. Inflate that ear to a two-face fan keyed at the other
root, and denote the graph after all these inflations by $\widetilde G$.
For this particular choice of demands, the normalized optimum on $G$ lifts
to $\widetilde G$ with the same seed number: each added interior vertex is
accompanied by one added arrow. Thus $Z(\widetilde G)\le Z(G)$. Conversely,
normalize any optimum on $\widetilde G$ and compress each fan to the
corresponding ear mode. The gluing data and seed number are preserved,
so $Z(G)\le Z(\widetilde G)$. Consequently
\[
 Z(\widetilde G)=Z(G).
\]
Keep these choices as $\theta$.

At each bare core side $ab$, attach the neutral island of
Theorem~\ref{thm:neutral}: a new core triangle $avb$ and terminal fans keyed
at $a$ and $b$. Let $b_0$ be the number of bare sides and $G'$ the resulting
graph. Attachments have disjoint new interiors; applying the neutral identity
in their actual successive exteriors gives
\[
 Z(G')=Z(\widetilde G)+b_0=Z(G)+b_0.
\]
The expanded core $C'$ is now fully decorated by strips with at least two
faces, so Theorem~\ref{thm:hard-full} provides an optimal protocol on $G'$
with all its hard capacities.

We compare those capacities vertex by vertex. Each old vertex lies on exactly
two sides of the core boundary. An old fan contributes a demand at its nonkey
and no demand at its key; an old serpent contributes a demand at both roots.
An inflated ear contributes exactly its chosen demand. Each formerly bare
side contributes a new fan keyed at the old endpoint and thus no demand
there. Hence the capacity at an old vertex in $C'$ is exactly $2-t_v$.
The new vertex $v$ in an island is the nonkey of both of its terminal fans,
so its capacity is zero. It has no selected core incidence. All terminal
root edges in $C'$ are unselected.

Restrict this optimal arrow set to $\widetilde G$. To track the cost, let
one removed island have $q$ vertices in its two terminal-fan interiors; with
its new core vertex it removes $q+1$ vertices. Each fan with $m$ interior
vertices has at most $m$ selected arrows, and the new core vertex has no
selected core arrow. Thus the total number of arrows removed by that island
is at most $q$. Its removal decreases $|V|-|M|$ by at least one. The
restricted incidence graph is a subgraph of the original one: for every
retained graph edge, its selected or unselected status is unchanged.
It is therefore still acyclic, and the restricted marking is executable.
After removing every island its seed number is at most
\[
 Z(G')-b_0=Z(\widetilde G).
\]
It is also at least $Z(\widetilde G)$ by optimality of that invariant, so
equality holds. All old core arrows survive, including any selected formerly
bare side, and their old capacities remain satisfied.

Finally compress each inflated fan to its chosen ear mode. This preserves
the root relation, occupations, normalized cost, and every old core arrow.
The resulting protocol is optimal on $G$, realizes the allocation $\theta$,
and satisfies the claimed capacities. Its underlying paths are induced by
Lemma~\ref{lem:incidence}; two selected edges in a core triangle would
contradict inducedness. This proves the face condition as well.
\end{proof}

\subsection{The corresponding undirected optimum}
For an ear allocation $\theta$ put
\begin{equation}\label{eq:allowed-J}
 \cJ(C,\theta)=
 \{J\subseteq E(C)\setminus\cR:
   \deg_J(v)\le c_v,\ |J\cap E(\Delta)|\le1\text{ for every core face}\}.
\end{equation}
Let $\nu(C)=\max_{\theta,J\in\cJ(C,\theta)}|J|$.

\begin{lemma}\label{lem:strip-undirected}
For a root-free strip with $m$ interior vertices, and root degree allowances $k_a,k_b\in\{0,1,2\}$, the maximum number of induced-forest edges is
\[
 m-1+\one_{k_b\ge1}\quad\text{for a fan keyed at }a,
\]
and
\[
 m-2+\one_{k_a\ge1}+\one_{k_b\ge1}\quad\text{for a serpent}.
\]
\end{lemma}
\begin{proof}
Here an unavailable root has selected degree zero. For a fan keyed at $a$,
the interior vertices form a path on $m$ vertices. Its $m-1$ edges are
always available. If the nonkey $b$ is available, attaching it to the end
$u_1$ gives $m$ edges and uses no degree at $a$. More than $m$ edges on the
$m+2$ vertices would force a single spanning path, which is impossible in
a graph containing a triangle. If $b$ is unavailable, it is a singleton.
When $a$ is used, the remaining $m+1$ vertices contain a triangle and require
at least two induced-path components, giving at most $m-1$ edges. When $a$
is also unused, the $m$ interior vertices have at most $m-1$ edges. These
bounds prove the fan formula.

For a serpent, the interior graph already contains a triangle by
Lemma~\ref{lem:serpent-geometry}. If precisely $q\in\{0,1,2\}$ roots are
available, the other $2-q$ roots are singletons. On the remaining $m+q$
vertices an induced linear forest has at least two components, hence at most
$m+q-2$ edges. The full two-root construction has $m$ edges and uses exactly
one selected edge at each root. Delete the incidence at every unavailable
root. A forest edge deletion preserves inducedness of each remaining path
and removes precisely $2-q$ edges, attaining $m+q-2$. Thus an allowance of
two at one root cannot improve the bound: one incidence at that root already
attains its entire possible contribution.
\end{proof}

\begin{theorem}\label{thm:path-capacity}
For the canonical core of a maximal biconnected outerplanar graph with branching weak dual,
\begin{equation}\label{eq:path-capacity}
 P(G)=|V(C)|-\nu(C).
\end{equation}
\end{theorem}
\begin{proof}
Begin with a maximum induced linear forest on $G$. If a terminal root $ab$
is selected, its restriction to a strip with $m$ interior vertices has at
most $m$ edges, since that strip contains a triangle. The root already uses
one degree at each endpoint. Replace the whole strip restriction, including
the selected root, by the two-root forest with $m$ edges and root degree at
most one. No external degree increases. The core face adjacent to $ab$
loses that selected edge, and every strip face satisfies the local condition.
Lemma~\ref{lem:linear-forest} therefore proves that the replaced global
selection is still an induced forest and has no fewer edges. Repeating gives
a maximum forest with all terminal roots unselected.

Every selected ear edge chooses its demanded root; an empty ear may be
allocated either way. Fix such an allocation $\theta$ and a core selection
$J$ satisfying degree at most two and the face conditions. At vertex $v$,
let $s_v=2-\deg_J(v)$ be its remaining slots and $t_v$ its demanded terminal
incidences. A demand needs one slot, and the minimum number of unsatisfied
demands is
\[
 (t_v-s_v)_+=(\deg_J(v)-c_v)_+.
\]
The choices of which demands receive slots can be made independently at
each core vertex. Lemma~\ref{lem:strip-undirected} then realizes every
resulting subset of available demanded roots, losing exactly one local edge
for each unfulfilled demand. An ear has the same rule with one demand.
The union satisfies every degree and face condition, so this local count
is attained by a global induced forest. Conversely those conditions and
the local upper bounds rule out any larger count. Hence the exact optimum
above $J$ is
\begin{equation}\label{eq:soft-penalty}
 \sum_i m_i+|J|-\sum_v(\deg_J(v)-c_v)_+.
\end{equation}
An arbitrary forest is included by choosing its actual ear allocation;
maximizing over $\theta$ therefore loses no optimum.

Write $g(J)=|J|-\sum_v(\deg_J(v)-c_v)_+$. If $v$ exceeds its capacity,
choose a selected edge $vw$. Deleting it reduces $|J|$ by one and reduces
the penalty at $v$ by one. At $w$ it reduces the penalty by one precisely
when $w$ was also over capacity. Thus
\[
 g(J\setminus\{vw\})-g(J)
 =-1+1+\mathbf1_{\deg_J(w)>c_w}\ge0.
\]
The face inequalities persist under deletion. The edge count decreases at
each step, so after finitely many steps we obtain a hard-capacity set $J'$
with $|J'|=g(J')\ge g(J)$. Conversely every hard-capacity set has zero
penalty. Maximizing~\eqref{eq:soft-penalty} therefore gives
$\sum_i m_i+\nu(C)$ selected edges. Since $G$ has
$|V(C)|+\sum_i m_i$ vertices, subtracting edges from vertices gives
\eqref{eq:path-capacity}.
\end{proof}
Theorems~\ref{thm:hard-general} and~\ref{thm:path-capacity} will now
be combined with the direction constraints on the resulting path forests.

\section{Path phases and the repair formula}\label{sec:phases}

\subsection{Blocking circuits}
Let $\cP=(P_1,\ldots,P_m)$ be any vertex-disjoint induced path cover of a finite simple graph. Fix a reference order on each nontrivial path. A sign $s_i\in\{1,-1\}$ follows or reverses this order. Singletons are initially black and have no sign variable.

For a nonterminal vertex $u$ in the chosen direction write $u^+$ for its successor. The vertex dependency graph has arcs
\begin{equation}\label{eq:vertex-dependency}
 w\longrightarrow u^+\qquad
          (w\in N[u]\setminus\{u^+\}).
\end{equation}
As in Lemma~\ref{lem:incidence}, the prescribed directions are executable if and only if this graph is acyclic. Dependencies within one induced path point forward.

\begin{definition}\label{def:blocking}
A blocking circuit consists of $k\ge2$ distinct nontrivial covering paths $P_{i_1},\ldots,P_{i_k}$, distinct vertices $a_j,b_j$ on each of them, and edges $b_ja_{j+1}$ for all $j$ modulo $k$. Let $\sigma_j=1$ if $a_j$ precedes $b_j$ in the reference order, and $-1$ otherwise. Its selected path segments and cross edges form a simple cycle in the original graph.
\end{definition}

The chain-twist criterion of Cameron, Janssen, Mathew, and
Zhang~\cite[Definition~2.2 and Theorem~2.2]{CJMZ} has the following
path-segment form.

\begin{theorem}\label{thm:blocking}
A direction choice is not executable if and only if some blocking circuit satisfies either $s_{i_j}=\sigma_j$ for every $j$, or $s_{i_j}=-\sigma_j$ for every $j$.
\end{theorem}
\begin{proof}
First suppose every segment of the circuit points from $a_j$ to $b_j$. The vertex $a_j$ then has a selected successor $a_j^+$ on its covering path. The contact edge $b_ja_{j+1}$ forces the prerequisite
\[
 b_j\longrightarrow a_{j+1}^+.
\]
Following the selected path from $a_{j+1}^+$ to $b_{j+1}$ gives a directed dependency path. Concatenating these paths around the circuit gives a directed closed walk, so no chronology can execute the arrows. If all segment directions are reversed, traverse the same circuit in the opposite cyclic order and use the same argument.

For the converse, begin with all path starts black and perform prescribed arrows for as long as possible, choosing any currently legal prescribed arrow. By assumption the process stops before every path is complete. On each unfinished path $P_i$, its black vertices form an initial segment in the chosen direction: the only prescribed incoming arrow of a noninitial vertex is from its predecessor. Let $a_i$ be its last black vertex and $a_i^+$ its next, white vertex.

The arrow $a_i\to a_i^+$ is not legal. Since $a_i$ is black and $a_i^+$ is white, there is a further white neighbour $w_i\ne a_i^+$. This neighbour cannot lie on $P_i$. A preceding vertex of $P_i$ is black, and an edge from $a_i$ to a later vertex other than its successor would be a chord of this induced path. Thus $w_i$ lies on another unfinished path; denote its index by $f(i)$. This constructs a map from the finite nonempty set of unfinished paths to itself, with $f(i)\ne i$.

An iteration of $f$ reaches a cycle of distinct indices. Relabel this cycle in reverse order as $i_1,\ldots,i_k$, so that $f(i_{j+1})=i_j$ cyclically. On $P_{i_j}$ put
\[
 b_j=w_{i_{j+1}}.
\]
The contact edge $b_ja_{i_{j+1}}$ exists by construction. The vertex $a_{i_j}$ is black and $b_j$ is white, so they are distinct and $a_{i_j}$ strictly precedes $b_j$ in the chosen direction of $P_{i_j}$. Hence the nonempty segments from $a_{i_j}$ to $b_j$, together with these contact edges, form a blocking circuit with every segment pointing in the same cyclic direction. This is one of the two forbidden sign patterns. Singleton paths were initially complete and never enter the construction.
\end{proof}

\subsection{The sign scheme}
Let $I=\{i:|P_i|>1\}$. Define
\[
 B_{\cP}=\CC[s_i:i\in I]/(s_i^2-1:i\in I).
\]
For a blocking circuit $\Gamma$ define two idempotents
\begin{equation}\label{eq:projectors}
 e_\Gamma^+=\prod_{j=1}^k\frac{1+\sigma_js_{i_j}}2,
 \qquad
 e_\Gamma^-=\prod_{j=1}^k\frac{1-\sigma_js_{i_j}}2.
\end{equation}
Set
\begin{equation}\label{eq:orientation-scheme}
 \cO_{\cP}=\Spec\bigl(B_{\cP}/(e_\Gamma^+,e_\Gamma^-:\Gamma\text{ blocking})\bigr).
\end{equation}

\begin{proposition}\label{prop:orientation-scheme}
The complex points of $\cO_{\cP}$ are exactly the executable directions of $\cP$. The scheme is finite, split, and reduced.
\end{proposition}
\begin{proof}
Evaluation on the sign vectors gives
\[
 B_{\cP}=\CC[s_i:i\in I]/(s_i^2-1)
 \cong\prod_{\varepsilon\in\{\pm1\}^I}\CC.
\]
One can see this directly from the mutually orthogonal idempotents
\[
 e_\varepsilon=\prod_{i\in I}\frac{1+\varepsilon_is_i}{2},
 \qquad \sum_\varepsilon e_\varepsilon=1.
\]
At a sign vector $\varepsilon$, each blocking projector $e_\Gamma^+$ or
$e_\Gamma^-$ is either zero or one. It is one exactly when all segments
of that circuit have one of its two forbidden orientations. Thus the ideal
generated by all projectors has the full factor $\CC$ at an excluded sign
and the zero factor at every other sign. Its quotient is the product of
$\CC$ indexed by the surviving sign vectors.

By Theorem~\ref{thm:blocking}, those survivors are exactly the executable
orientations of the prescribed paths. Their starts and the singleton paths
are the corresponding initial set. The quotient product is finite, split,
and reduced, and is zero precisely when no executable orientation exists.
\end{proof}

\subsection{The two-path cohomology class}
Suppose two reference paths are $x_0,\ldots,x_p$ and $y_0,\ldots,y_q$. Two cross edges $x_ay_b,x_cy_d$ with $a\ne c,b\ne d$ give
\begin{equation}\label{eq:contact-sign}
 \kappa=\sgn((c-a)(d-b)).
\end{equation}
The circuit follows $x_a$ to $x_c$ and then $y_d$ to $y_b$, so the segment sign product is $-\kappa$. Its two cyclic directions are excluded precisely by
\begin{equation}\label{eq:pair-sign}
 s_i s_j=\kappa.
\end{equation}
The two associated projector generators generate the same ideal as $s_is_j-\kappa$: their sum is $(1-\kappa s_is_j)/2$, and each is this sum multiplied by a one-path sign projector. In the induced subgraph on just these two paths, all blocking circuits have this form, so these equations are also sufficient there.

For each pair of cross edges $x_ay_b,x_cy_d$ with $a\ne c$ and $b\ne d$,
place an edge between the corresponding path vertices and label it by
$\sgn((c-a)(d-b))$. The resulting signed multigraph is $\Lambda_{\cP}$.
Identical repeated constraints may be merged; parallel edges with opposite
labels are retained. Write $\kappa_e=(-1)^{b_e}$ and
$s_i=(-1)^{\epsilon_i}$.

On a graph $\Lambda$, the constant cellular sheaf $\underline{\FF_2}$ has stalk $\FF_2$ at each vertex and edge, with identity incidence maps. Its cochain complex is
\[
 C^0(\Lambda;\FF_2)\xrightarrow{\delta}C^1(\Lambda;\FF_2),
 \qquad (\delta\epsilon)_{ij}=\epsilon_i+\epsilon_j.
\]
Here $H^1$ means the cokernel of this map. Equations~\eqref{eq:pair-sign} are exactly $\delta\epsilon=b$.

\begin{proposition}\label{prop:pair-cohomology}
All two-path constraints are simultaneously soluble if and only if
\[
 [b]=0\quad\text{in }H^1(\Lambda_{\cP};\FF_2).
\]
Equivalently every closed walk has label product $1$. If soluble, the solution set has $2^c$ points, where $c$ is the number of connected components, including isolated nontrivial paths. The class is independent of the reference path directions.
\end{proposition}
\begin{proof}
Write $\kappa_e=(-1)^{b_e}$ and $s_i=(-1)^{\epsilon_i}$. Then
\[
 s_is_j=\kappa_e
 \quad\Longleftrightarrow\quad
 \epsilon_i+\epsilon_j=b_e\quad\text{in }\FF_2.
\]
The full system is precisely $\delta\epsilon=b$ in the displayed cochain complex. Since the graph has no two-dimensional cells, every one-cochain is a cocycle, and the system has a solution exactly when $b$ belongs to $\operatorname{im}\delta$, or $[b]=0$ in its cokernel.

To make the criterion constructive, choose a spanning forest, including all vertices, and fix a value of $\epsilon$ at one root per component. Along a tree edge $ij$, set $\epsilon_j=b_{ij}+\epsilon_i$. The unique tree path gives a unique value at every vertex. For a nontree edge $uv$, the propagated values satisfy its equation exactly when the sum of $b$ along that edge and the tree path from $u$ to $v$ is zero. Multiplicatively this is the product-one condition around its fundamental closed walk. A parallel edge is a nontree edge of this same calculation and cannot be dropped when its sign differs. Every other closed walk is a sum of such cycles, so these tests are sufficient as well as necessary.

If a solution exists, the difference of two solutions lies in $\ker\delta$, consisting of the functions constant on each component. There are $2^c$ such functions, including the independent value on each isolated nontrivial path. Finally, reversing a reference path changes the signs of all incident constraints. It replaces $b$ by $b+\delta\chi_i$, where $\chi_i$ is the indicator of that path vertex. Hence its class is unchanged.
\end{proof}
The specified class $[b]$ is the obstruction to trivializing the local $\mu_2$ transition data.

If the two-path equations are soluble, choose a solution $a_i\in\{1,-1\}$. All solutions have the form
\[
 s_i=a_i z_{c(i)},\qquad z_c^2=1,
\]
where $c(i)$ is the component of $\Lambda_{\cP}$. For a longer blocking circuit put $\nu_j=\sigma_ja_{i_j}$. If opposite $\nu_j$ occur in the same component, both substituted projectors vanish identically, since they contain $(1+z_c)(1-z_c)$. Otherwise repeated factors are idempotent, and the remaining equations are precisely
\begin{equation}\label{eq:residual}
 \prod_{c\in C(\Gamma)}\frac{1+\nu_cz_c}2=0,
 \qquad
 \prod_{c\in C(\Gamma)}\frac{1-\nu_cz_c}2=0.
\end{equation}
A residual circuit involving one component is inconsistent; two components
produce another relation of the form~\eqref{eq:pair-sign}; every larger
support contributes its displayed projector.

Let $T_{\cP}$ be the simple graph on the nontrivial covering paths, with
an edge $ij$ whenever an edge of the original graph joins $P_i$ to $P_j$.
If $T_{\cP}$ is a forest, every blocking circuit uses exactly two covering
paths: a circuit on three or more distinct paths would project to a cycle
in $T_{\cP}$. In this case the equations of
Proposition~\ref{prop:pair-cohomology} are sufficient for executability.

For example the graph with faces $012,013,124,205$ has the cover $(0,3),(1,4),(2,5)$. There are no two-path constraints, but the circuit $0,3,1,4,2,5,0$ gives
\[
 (1+s_0)(1+s_1)(1+s_2)=0,
 \qquad (1-s_0)(1-s_1)(1-s_2)=0.
\]
Here $\Lambda_{\cP}$ has no edges, whereas $T_{\cP}$ is a triangle.
The full scheme has six points and the two-path scheme has eight.

\subsection{Designated cuts and the main formula}\label{sec:phase-repair}
Retain the fixed full fan and serpent forests from
Section~\ref{sec:strips}, and denote this family by $\mathcal F$. Fix
$\theta$ and $J\in\cJ(C,\theta)$ from~\eqref{eq:allowed-J}. In each ear
use the edge joining its interior vertex to $\theta_i$. Their unoriented
union with $J$ is denoted by $L(J)$.

Each local full forest has $m_i$ selected edges. At a core vertex
\[
 \deg_{L(J)}(v)=\deg_J(v)+t_v\le2,
\]
and every triangular face contains at most one selected edge. Thus Lemma~\ref{lem:linear-forest} gives a spanning induced forest with
\begin{equation}\label{eq:full-cover-count}
 |\cP(L(J))|=|V(C)|-|J|.
\end{equation}

Let $\cA_\theta$ consist of the selected edges connecting each demanded root to the interior of its strip. Each such edge has one endpoint at a root and the other in the strip interior. For $D\subseteq\cA_\theta$, set
\[
 L(J,D)=L(J)\setminus D,
 \qquad \cP(\theta,J,D)=\text{its induced path components}.
\]
Deleting a forest edge creates one further component, so
\begin{equation}\label{eq:cut-count}
 |\cP(\theta,J,D)|=|V(C)|-|J|+|D|.
\end{equation}
Apply~\eqref{eq:orientation-scheme} to this particular cover and write its scheme as $\cO_{\theta,J,D}$. Define
\begin{equation}\label{eq:tau}
 \tau_C(\theta,J)=
 \min\{|D|:D\subseteq\cA_\theta,\ \cO_{\theta,J,D}\ne\varnothing\},
\end{equation}
with $+\infty$ for an empty minimum. For each $D$, feasibility is given
by the blocking equations~\eqref{eq:projectors}.

\begin{theorem}\label{thm:main}
Let $G$ be a maximal biconnected outerplanar graph whose weak dual has a degree-three node, and let $C$ be its canonical core. Then
\begin{equation}\label{eq:main}
 \boxed{Z(G)=r-\max_{\theta,J\in\cJ(C,\theta)}
                  \bigl(|J|-\tau_C(\theta,J)\bigr),\qquad r=|V(C)|.}
\end{equation}
If the weak dual has $\ell$ leaves and $d_2$ degree-two nodes retained in its canonical core, then equivalently
\[
 Z(G)=\ell-\max_{\theta,J}(|J|-d_2-\tau_C(\theta,J)).
\]
For a weak-dual path, $P(G)=Z(G)=2$.
\end{theorem}
\begin{proof}
Let $r=|V(C)|$. For fixed $\theta,J,D$, the full branch forests have $\sum_i m_i$ edges and the added core contributes $|J|$ edges. Their root degrees are at most two by the hard capacities, and every face has at most one selected edge. Thus their union is an induced forest. Deleting $D$ leaves exactly
\[
 r-|J|+|D|
\]
components. If $\cO_{\theta,J,D}$ has a point, orient these components according to that point. Proposition~\ref{prop:orientation-scheme} gives an actual executable arrow set. Its component starts are a zero-forcing set of that cardinality. Minimizing over $D$, and then over $\theta,J$, proves
\[
 Z(G)\le\min_{\theta,J}\bigl(r-|J|+\tau_C(\theta,J)\bigr).
\]

For the opposite inequality, apply Theorem~\ref{thm:hard-general} to obtain an optimal protocol with capacity-allowed core set $J$ and a choice $\theta$ for the ears. Its branches are in the table modes. Each such mode is realized by the fixed full forest of that same branch: a fan uses its nonkey path or deletes its root incidence; a serpent reverses both full root chains if necessary and deletes precisely the inactive root incidences; an ear uses its fixed demanded edge or deletes it. Lemma~\ref{lem:modes} proves that these realizations have exactly the recorded occupation, relation and cost. Hence replacing the local representatives by these particular realizations preserves global compatibility and the seed count.

Let $D$ be the union of the deleted root incidences. Together with the
unchanged core arrows, the selected forest is $L(J,D)$. Its actual directions are a point of $\cO_{\theta,J,D}$ and satisfy
\[
 Z(G)=r-|J|+|D|\ge r-|J|+\tau_C(\theta,J)
 \ge\min_{\theta',J'}\bigl(r-|J'|+\tau_C(\theta',J')\bigr).
\]
The two inequalities prove the formula. The normalized optimum supplies
a finite term in the minimum over the finite set of $\theta,J$.

Lemma~\ref{lem:core} gives $r=\ell+d_2$, proving the equivalent leaf formula. If the weak dual has no degree-three node, root it at an end face: the whole graph is a terminal strip and two black root vertices force it. Minimum degree two excludes one seed; a triangle excludes an induced one-path cover. Hence $P(G)=Z(G)=2$ in that case.
\end{proof}

\begin{corollary}\label{cor:gap}
With $\nu(C)$ as in Theorem~\ref{thm:path-capacity},
\[
 Z(G)-P(G)=\min_{\theta,J}\bigl(\nu(C)-|J|+\tau_C(\theta,J)\bigr).
\]
In particular $Z(G)=P(G)$ exactly when some maximum-cardinality allowed core selection has a nonempty phase scheme without root deletions.
\end{corollary}
\begin{proof}
Subtract~\eqref{eq:path-capacity} from~\eqref{eq:main}. Both summands inside the resulting minimum are nonnegative. Their sum is zero precisely when both vanish.
\end{proof}

\begin{lemma}\label{lem:repair-monotone}
For fixed $\theta$ and fixed full forests, if $J'\subseteq J$ and $D'\supseteq D$, then
\[
 \cO_{\theta,J,D}\ne\varnothing\ \Longrightarrow\
 \cO_{\theta,J',D'}\ne\varnothing.
\]
Consequently $\tau_C(\theta,J')\le\tau_C(\theta,J)$.
\end{lemma}
\begin{proof}
Delete from an executable orientation the arrows on $J\setminus J'$ and $D'\setminus D$. Lemma~\ref{lem:deletion} preserves executability, and the remaining underlying forest is exactly $L(J',D')$. For the inequality take $D'=D$ and minimize.
\end{proof}
The point count can decrease under deletion: a single edge has two
directions, while its two singleton components have one direction choice.
The net gain $|J|-\tau_C(\theta,J)$ compares the change in core size
with the change in repair cost.

\subsection{The finite capacity--phase scheme}
For fixed $\theta$, use a Boolean variable $z_e$ for each $e\in E(C)\setminus\cR$. Impose
\begin{equation}\label{eq:boolean-capacity}
 \begin{gathered}
 z_e^2-z_e=0,\\
 \prod_{e\in A}z_e=0\quad
 (A\subseteq\delta_C(v)\setminus\cR,\ |A|=c_v+1),\\
 z_ez_f=0\quad
 (e\ne f,\ e,f\in E(\Delta)\setminus\cR).
 \end{gathered}
\end{equation}
Its points are exactly $\cJ(C,\theta)$. The coordinate algebra is a quotient of a finite product of fields by idempotent constraints, hence remains finite, split, and reduced.

Over each such point retain the cut sets and their sign schemes. Their disjoint union is
\[
 \mathfrak Y_{G,\mathcal F}=\bigsqcup_{\theta,J,D}\cO_{\theta,J,D}.
\]
Give a point weight $r-|J|+|D|$, and define
\begin{equation}\label{eq:weighted-count}
 Q_{G,\mathcal F}(t)=\sum_{\theta,J,D}
       |\cO_{\theta,J,D}(\CC)|\,t^{r-|J|+|D|}.
\end{equation}
Every coefficient is a nonnegative integer. Theorem~\ref{thm:main} is therefore equivalent to
\[
 \ord_{t=0}Q_{G,\mathcal F}(t)=Z(G).
\]
For every permitted $\mathcal F$, the least weight is therefore $Z(G)$.
The coefficient of $t^k$ counts tuples $(\theta,J,D,s)$ with
$s\in\cO_{\theta,J,D}(\CC)$ and $r-|J|+|D|=k$. Equivalently,
$Q_{G,\mathcal F}$ is the Hilbert series of the vector space with this
weighted basis. Several such tuples can have the same initial set.

For comparison with graph matrices, take
\[
 R_G=\mathbb Q[d_v,t_e^{\pm1}],\qquad
 \mathbf A_{uv}=d_u\ (u=v),\quad t_{uv}\ (u\sim v),\quad0\text{ otherwise}.
\]
The local implication used by a force is
\[
 0=\sum_{w\in N[u]}\mathbf A_{uw}x_w
   =t_{uv}x_v\quad\Longrightarrow\quad x_v=0.
\]
Let $S$ be the initial set, put $T=V\setminus S$, and let $U$ be the set
of force sources. Order $T$ and $U$ by the same chronology. The submatrix
$\mathbf A_{U,T}$ is lower triangular, with diagonal entries
$t_{uv}\in R_G^\times$: at the force $u\to v$, every other neighbour of
$u$ has already become black. In particular, this submatrix is invertible
over $R_G$.

The equation complex with the seed coordinates set to zero is
\[
 R_G^T\xrightarrow{\ \mathbf A_{V,T}\ }R_G^V,
\]
in degrees zero and one. Put $\mathbf D=\mathbf A_{U,T}$ and let $\mathbf E$ consist of the
remaining rows. In this row order,
\[
 \mathbf A_{V,T}=\begin{pmatrix}\mathbf D\\\mathbf E\end{pmatrix},
 \qquad
 \begin{pmatrix}
 \mathbf D^{-1}&0\\-\mathbf E\mathbf D^{-1}&I
 \end{pmatrix}
 \begin{pmatrix}\mathbf D\\\mathbf E\end{pmatrix}
 =\begin{pmatrix}I\\0\end{pmatrix}.
\]
Thus the complex splits as $|T|$ identity complexes and a free module
$R_G^{|S|}$ in degree one. Its differential is split injective and its
cokernel is free of rank $|S|$.

The universal module and its flatness characterization are treated
in~\cite{JLDSL}. The finite scheme above records the combinatorial witnesses
used for the present normal form. The companion paper~\cite{HLgeneral} develops their graph-matrix setting, and~\cite{HLcomparison} identifies the phase algebra with the locally constant functions on the corresponding clock realization space.

\section{Exact replacement of rooted graph blocks}\label{sec:replacement}

The response calculus also describes when one graph block can replace
another in a larger graph. We first keep the exterior protocol fixed,
which gives terminal and corridor reductions. We then allow the exterior
to reoptimize and obtain the finite additive context quotient.

\subsection{Terminal and corridor replacement}
The raw edge vector is finite on only three states, with values $0,-1,-1$.
For the operations $L(H)=T(H,E)$ and $R(H)=T(E,H)$ defined in
Section~\ref{sec:responses}, fixing the corresponding child in
\eqref{eq:triangle-F} to this edge vector gives min-plus matrices $K_L,K_R$.
Thus
\[
 r_{L(H)}=K_L\otimes r_H,\qquad r_{R(H)}=K_R\otimes r_H.
\]
Both matrices have 120 finite entries, all zero or one.

\begin{lemma}\label{lem:reciprocity}
\[
 \mathsf C\otimes K_L=K_L^{\mathsf T}\otimes\mathsf C,
 \qquad \mathsf C\otimes K_R=K_R^{\mathsf T}\otimes\mathsf C.
\]
Hence $\beta_{L(H)}=K_L^{\mathsf T}\otimes\beta_H$, and similarly for $R$.
\end{lemma}
\begin{proof}
Fix a state $s$ on one edge of a triangle, a prescribed exterior state $t$ on another, and leave the third edge as the auxiliary one-edge block. Each side of the matrix identity enumerates the same assignments of arrows to this triangle and to the fixed root protocols. The compatibility test is in both cases that their union occupations are at most one and that the same six-port graph is acyclic.

Let $e$ denote the freely chosen third edge and $f$ the edge closed after
the first triangle composition. The local contribution in $\mathsf C\otimes K_L$ is
\[
 -\eps(e)+1-\eps(f)+\eps(f)=1-\eps(e).
\]
If the triangle is cut along its other edge first, the closed edge changes, but its subtraction in the triangle cost is again cancelled by its addition in $\mathsf C$. Thus the contribution in $K_L^{\mathsf T}\otimes\mathsf C$ is also $1-\eps(e)$. The assignments are identical, with no change in either source or target occupation. Minimizing their equal weights proves the identity in every $(t,s)$ entry, including entries with no admissible assignment. Reflection gives the $R$ identity. Associativity of minimum-plus matrix multiplication then gives
\[
 \beta_{L(H)}=\mathsf C\otimes K_L\otimes r_H
 =K_L^{\mathsf T}\otimes \mathsf C\otimes r_H
 =K_L^{\mathsf T}\otimes\beta_H.
\]
\end{proof}

For $s=(i_a,i_b,o_a,o_b;R)$ define
\[
 f_a(s)=\max\{i_bo_b,
 i_b\one_{q_aRq_b},\ o_b\one_{p_bRp_a}\}.
\]
Let $d_a(s)=1$ precisely when $i_a=o_a=1$, or when the occupation and relation are respectively
\[
 (0,1,1,0),\ R\in\{R_0,C_+\};\qquad
 (1,0,0,1),\ R\in\{R_0,R_-\}.
\]
Otherwise $d_a=0$. Reflection defines $f_b,d_b$. The six terminal responses are
\begin{equation}\label{eq:six-responses}
 0,\quad \min(f_a,f_b),\quad f_a,\quad f_b,\quad f_a+d_a,\quad f_b+d_b,
\end{equation}
labelled $E,T,F_a,F_b,S_a,S_b$ respectively. Their complete unary table is
\[
\begin{array}{c|cc}
 &L&R\\\hline
 E&T&T\\ T&F_a&F_b\\ F_a&F_a&S_b\\ F_b&S_a&F_b\\
 S_a&S_a&S_b\\ S_b&S_a&S_b.
\end{array}
\]
The six primitive graph responses follow by enumerating their $3,7,19,19,53,53$ executable arrow sets. Every entry of the displayed table is a min-plus expression from Lemma~\ref{lem:reciprocity}; the complete 672-component arithmetic certificate lists every finite candidate and an attaining minimum. Induction on the strip length then proves~\eqref{eq:six-responses} for all terminal strips. In particular a fan of at least two faces may be replaced, in every prescribed exterior protocol, by its two-face representative; a serpent by the corresponding three-face representative. These identities give equality of the conditional response vectors used in~\eqref{eq:response-gluing}.

\begin{proposition}\label{prop:fan-gap}
For every compatible rooted exterior $Q$,
\[
 Z(Q\cup S_a)-Z(Q\cup F_a)\in\{0,1\}.
\]
The difference is one exactly when $d_a(t)=1$ at every minimizer of $r_Q(t)+f_a(t)$.
\end{proposition}
\begin{proof}
The difference equals
\[
 \min_t(r_Q(t)+f_a(t)+d_a(t))-
 \min_t(r_Q(t)+f_a(t)).
\]
Use $d_a\in\{0,1\}$. If every old minimizer has penalty one, every nonminimizer already has an integer cost at least one larger, so the minimum rises by one. Otherwise an old minimizer remains unchanged.
\end{proof}

\begin{proposition}\label{prop:corridors}
Every degree-two weak-dual corridor may be replaced by one with at most ten triangular faces, preserving its entire min-plus action on arbitrary 56-dimensional input cost vectors and preserving endpoint identifications. If a maximal outerplanar graph has $\ell\ge3$ weak-dual leaves, successive replacements produce a graph with the same zero forcing number and at most $14\ell-30$ vertices.
\end{proposition}
\begin{proof}
If $w=c_1\cdots c_m$ records the successive operations from the input root
to the output root, its action is
\[
 K_w=K_{c_m}\otimes\cdots\otimes K_{c_1}.
\]
The finite certificate consists of 209 distinct matrices, including the
identity, a representative word for each matrix, and all 418 products
by a generator on the left. Their shortest recorded word counts in lengths $0$ through $10$ are
\[
 1,2,4,8,16,28,44,46,34,20,6.
\]
Every generator product is recomputed from the matrices. Closure starting from the identity proves that every word has one of these matrices. A root endpoint survives to the other end exactly when one of the two letters never occurs; the representatives preserve this datum. Equal matrices act identically on every input vector, so they may be substituted between arbitrary compatible end blocks.

After suppressing degree-two dual paths, the dual has $\ell-2$ branch nodes, $\ell$ terminal paths, and $\ell-3$ internal corridors. Keep the branch triangles, replace terminal strips by representatives with at most three faces, and corridors by representatives with at most ten. The face count is at most
\[
 (\ell-2)+3\ell+10(\ell-3)=14\ell-32.
\]
A polygon triangulation has two more vertices than faces. If the original weak dual is a path, the graph has $Z=2$ and can instead be replaced by a triangle.
\end{proof}

\subsection{Reoptimizing the exterior}\label{subsec:free}
Equal responses to prescribed protocols are stronger than equal optima against freely reoptimized exteriors.

\begin{proposition}\label{prop:deficit}
A one-root incoming serpent mode can be filled to the two-root incoming mode after deleting at most the exterior arrow entering its missing root. Its outgoing counterpart deletes the corresponding exterior outgoing arrow. An empty serpent can be filled after at most two such deletions. None of these operations increases the seed count.
\end{proposition}
\begin{proof}
After deleting the requisite incoming arrows, both root column ports are exterior sources. A full incoming serpent adds only a column-to-column relation, which cannot close a path entering either source. Local arrows increase by the number of filled root incidences; exterior arrows decrease by at most the same number. Use root rows as sinks for the outgoing argument.
\end{proof}

\begin{theorem}\label{thm:serpent-reflection}
For arbitrary finite simple exteriors $Q$, the two root-reflected serpent types have the same optimized exterior value:
\[
 Z(Q\cup S_a)=Z(Q\cup S_b).
\]
The strips may have arbitrary lengths and later turns.
\end{theorem}
\begin{proof}
Normalize an optimum on $Q\cup S_a$ by Theorem~\ref{thm:terminal-normal}. Its root is unselected and it uses one of the seven modes. We construct a protocol on $Q\cup S_b$ of no greater seed cost; reversing the roles will give equality.

The empty mode has cost two and relation $R_0$ in either strip. For the full inward mode, neither root column has an incoming exterior incidence. Hence both are exterior sources. The two serpent conventions add opposite column relations, but either relation between these sources is harmless: a cycle using it would require a path entering its source. The full outward case is the same argument with two exterior row sinks. These full modes have cost zero.

Among the one-root modes, those at the old single-neighbour root merely lose an extra relation under reflection. Keeping the same physical root and direction thus preserves occupation and cost and is safe. In the remaining inward case, the new mode adds $p_b\to p_a$ whereas the old mode had no extra relation. Adding this pair creates a cycle exactly when the exterior has a path $p_a\to p_b$. If there is no such path, use the one-root mode directly.

If that exterior path exists, its last incidence into $p_b$ must be the unique selected exterior arrow entering $b$, say $u\to b$. Delete it and use the full inward mode of $S_b$. The root $a$ already had no exterior incoming arrow, and after the deletion neither does $b$. Both columns are now exterior sources, so the new mode is compatible. The exterior loses one arrow; the local cost drops from one to zero. Consequently the total number of seeds is unchanged, even when the two strips have different numbers of internal vertices. The outward case is obtained by transposition and exchanging $p,q$: delete the exterior arrow leaving the missing root and use the full outward mode.

Lemma~\ref{lem:modes} realizes each mode on its original strip. The
normalized cost $m-|M_H|$ gives the same comparison for different strip
sizes. This proves $Z(Q\cup S_b)\le Z(Q\cup S_a)$. Applying the same construction with the two conventions interchanged proves the reverse inequality.
\end{proof}

For rooted polygons, including $K_2$, define $H\sim H'$ by
\begin{equation}\label{eq:free-equivalence}
 Z(H\cup Q)-Z(H)=Z(H'\cup Q)-Z(H')
 \quad\text{for every ordered-root polygon }Q.
\end{equation}
Define $\xi(H,Q)=Z(H\cup Q)-Z(H)-Z(Q)$. If $i,j$ are conditional types, then
\begin{equation}\label{eq:free-matrix}
 \xi(H,Q)=g_{ij}:=
 2+\min_s(x_j(s)+b_i(s))-\delta_i-\delta_j.
\end{equation}
The scalars $\lambda_H,\lambda_Q$ cancel. Hence equality in~\eqref{eq:free-equivalence} is exactly equality of complete rows of $g$.

\begin{theorem}\label{thm:284}
There are exactly 284 equivalence classes under~\eqref{eq:free-equivalence}. Seventeen actual exterior polygons, each with at most 15 vertices, distinguish all classes. The classes form a weighted congruence for both slots of triangle composition. If their types are $i,j$, then
\[
 \operatorname{type}(T(H,K))=\nu(i,j),\qquad
 Z(T(H,K))=Z(H)+Z(K)+\kappa(i,j),
\]
where $\kappa\in\{-2,-1,0,1\}$. Therefore a triangle tree with $f$ faces satisfies
\[
 Z(G)=f+1+\sum_{\Delta}\kappa(i_{\Delta,L},i_{\Delta,R}).
\]
The number 284 is minimal for an exact replacement system consisting of one intrinsic type and one additive scalar with arbitrary rooted polygon exteriors.
\end{theorem}
\begin{proof}
The complete computation of~\eqref{eq:free-matrix} groups the 568 rows into 284 distinct rows. Every old type has an actual graph representative, so row equality is necessary as well as sufficient. The supplied seventeen test columns still separate all $\binom{284}{2}=40,186$ pairs; all their actual graph witnesses are checked.

To prove congruence, fix a parent exterior and group it with the triangle and the other child as an exterior for the child being replaced. Equation~\eqref{eq:free-equivalence} changes the parent optimum by $Z(H)-Z(H')$. Taking the parent exterior to be $K_2$ gives the same change in its intrinsic $Z$, proving equivalence of the parents. The quotient operation is therefore well-defined. Explicitly it is obtained from the old product by
\[
 \kappa([a],[b])=\gamma(a,b)+\delta_{\mu(a,b)}-\delta_a-\delta_b.
\]
The complete certificate checks all old representative products for consistency and all $284^2=80,656$ quotient products. Its finite weights have the stated range. A binary tree with $f$ internal nodes has $f+1$ edge leaves, each of value one, proving the sum formula.

Finally any exact scalar-weighted replacement that identifies two of these distinct classes would make their values differ by a constant in every exterior. The exterior $K_2$ forces that constant to be $Z(H)-Z(H')$, contrary to their distinction in~\eqref{eq:free-equivalence}. Thus every such replacement system has at least 284 types. 
\end{proof}

\section{Exact families and forcing obstructions}\label{sec:examples}

\subsection{A cyclic family of terminal strips}
Let $C$ be any triangulation of an $r$-gon, $r\ge3$, and attach a terminal strip along every outer side. All serpent strips may have arbitrary lengths and turns. Fans have at least two faces. Denote by $G_S$ the graph with all $r$ strips serpents, and by $G_F$ the graph obtained by changing exactly one of them to a fan.

\begin{lemma}\label{lem:backward-strip}
Let a terminal strip have $m\ge2$ triangular faces and original root $ab$. Its last added vertex $u_m$ and the preceding added vertex $u_{m-1}$ form a two-element forcing set of the strip. There is a chronology which uses no original root as a source before all its interior vertices and both roots are black. Consequently the same arrows are legal in any exterior in which the other neighbours of $a,b$ remain white.
\end{lemma}
\begin{proof}
The edge $u_{m-1}u_m$ lies on the last face and on the outer boundary. Root the same strip at this edge and read its path of faces in reverse order. Each new face in this order meets the already read faces in one diagonal and adds a vertex on the other side of that diagonal. Thus it is again the terminal-strip construction, now starting from $u_{m-1},u_m$. Apply the retiring-endpoint procedure of Section~\ref{sec:strips}.

It remains to check that this procedure need not use an original root as a source. The original root $b$ lies only in the first original face, which is the last face in reverse order, so it is the final vertex to be added. The original faces containing $a$ form an initial segment of the original face path, because once $a$ leaves a front it never returns. In reverse order, they therefore form a final segment. Once $a$ appears in the reverse construction, it belongs to every remaining front until the last face. It is never the retiring endpoint at an earlier step. At the final step the shared edge is $\{a,u_1\}$; choose the interior endpoint $u_1$ to force $b$. All preceding retiring sources were interior as well. Hence every used source has the same neighbours in the strip and in any attached exterior, so every step remains legal there. Both original roots have become black by the end.
\end{proof}

\begin{theorem}\label{thm:cyclic-strips}
\[
 P(G_S)=r,\qquad Z(G_S)=r+(r\bmod2),
 \qquad P(G_F)=Z(G_F)=r.
\]
\end{theorem}
\begin{proof}
For all serpents every core capacity is zero. Hence $J=\varnothing$ and Theorem~\ref{thm:path-capacity} gives $P=r$. With one fan there is a single positive-capacity core vertex, so again no core edge is allowed and $P=r$.

Suppose $G_S$ has an $r$-seed protocol. The hard-capacity normal form has $J=\varnothing$ and cost
\[
 r+\sum_i w_i=r.
\]
Thus each serpent has zero cost and is either $\Iroot\Iroot$ or $\Oroot\Oroot$. At every core vertex its two incident branches must occupy one incoming and one outgoing slot. If $b_i\in\{0,1\}$ records the choice on branch $i$, then
\[
 b_{i-1}+b_i=1\qquad(i\bmod r).
\]
For odd $r$ this is impossible. Consequently $Z(G_S)\ge r+1$ in that case.

For the upper construction, use the two internal seeds and the chronology of Lemma~\ref{lem:backward-strip} in each selected branch. Its sources are internal, so it can be completed without a condition on the white neighbours of its roots outside that branch. If $r$ is even, put this pair of seeds in alternating branches. There are $r$ seeds. These branches make every core vertex black. Each unfinished branch has both adjacent branches complete, so both its roots have all exterior neighbours black and the two-root construction finishes it. If $r$ is odd, choose $(r+1)/2$ alternating branches, covering every core vertex, and use $r+1$ seeds.

For one fan and odd $r$, label its root $c_0c_1$ with key $c_0$. Put one seed at $c_0$ and pairs in branches $1,3,\ldots,r-2$. This uses $1+2(r-1)/2=r$ seeds. The selected branches make all core vertices black. Root $c_1$ first clears the fan, since its other branch is complete and the fan needs only this nonkey root. All remaining serpents now have their exterior root neighbourhoods black and can be cleared. The even case uses alternating branch pairs as before. The path-cover lower bounds finish the proof.
\end{proof}
Theorem~\ref{thm:cyclic-strips} extends the one-separator-triangle
calculations in~\cite{HRR,DK} to arbitrary triangulated cores with these
terminal attachments.

For the zero-penalty configurations of the all-serpent family, let
$\varepsilon_i=1$ denote the outward mode of branch $i$ and
$\varepsilon_i=-1$ its inward mode. The occupation conditions are
\[
 \varepsilon_i^2=1,\qquad
 \varepsilon_i\varepsilon_{i-1}=-1\qquad(i\bmod r).
\]
These branch signs determine the directions of the full covering paths.
The system has two solutions for even $r$ and none for odd $r$.

\smallskip\noindent\textit{Remark.}
The same sign cocycle can be realized on a cycle of $r$ projective lines,
with the trivial bundle on each component and transition multiplier $-1$
at each node. Its product around the cycle is $(-1)^r$. The resulting
degree-zero line bundle is trivial for even $r$ and has order two for odd
$r$. Partial normalization at one node removes the cycle and its
obstruction.

\subsection{Two selected diagonals through one key}
For $k\ge3$ put $r=2k$. Triangulate the core by
\[
 (c_0,c_i,c_{i+1}),\qquad1\le i\le r-2.
\]
On every outer core edge $c_ic_{i+1}$ add $x_i,y_i$ and the faces
\[
 (c_i,c_{i+1},x_i),\qquad(\kappa_i,x_i,y_i),
\]
where $\kappa_i$ is the even-indexed endpoint. Call this graph $G_k$.

\begin{proposition}\label{prop:double-capacity}
The graph $G_k$ has $6k$ vertices, $2k$ weak-dual leaves, and
\[
 P(G_k)=Z(G_k)=2k-2.
\]
The graph formed by the eligible diagonals $c_0c_{2j}$ has ordinary
matching number one, whereas the capacity-allowed selection number is two.
\end{proposition}
\begin{proof}
Every even core vertex is the key of two fans and has capacity two; every odd one has capacity zero. Eligible positive-capacity diagonals are $c_0c_{2j}$, $1\le j\le k-1$. The capacity at $c_0$ bounds their selection by two, attained by $c_0c_2,c_0c_4$, which share no triangular face. Thus $P=2k-2$ by Theorem~\ref{thm:path-capacity}. The same star has ordinary matching number one.

For a forcing set take
\[
 S=\{c_4,c_6,\ldots,c_{r-2}\}
   \cup\{y_0,y_3\}\cup\{y_4,y_6,\ldots,y_{r-2}\}.
\]
It has $2k-2$ vertices. Execute
\[
 y_3\to x_3\to c_3,
 \qquad y_i\to x_i\to c_{i+1}\quad(i=4,6,\ldots,r-2).
\]
Their keys were initially black. Both fans keyed at $c_4$ are now complete, and its only white neighbour is $c_0$. Continue with
\[
 c_4\to c_0,\quad y_0\to x_0\to c_1,\quad
 c_{r-1}\to x_{r-1}\to y_{r-1},\quad c_0\to c_2.
\]
At the last step both fans keyed at $c_0$ are complete and $c_2$ is its only white core neighbour. All core vertices are now black. Every remaining fan can be forced from its odd nonkey endpoint, whose other fan is already complete. This gives the upper bound.
\end{proof}
A minimum induced path cover consists of $c_2-c_0-c_4$, the $k$ paths
\[
 y_{2j}-x_{2j}-c_{2j+1}-x_{2j+1}-y_{2j+1},\qquad0\le j<k,
\]
and the singleton paths $\{c_6\},\{c_8\},\ldots,\{c_{r-2}\}$. The selected two-edge path through the common key is essential to the saving. Replacing the capacity-two condition by an ordinary matching condition loses one unit.

\subsection{A two-path obstruction on seventeen vertices}\label{subsec:seventeen}
Consider the 17-vertex triangulation with core faces $(0,1,4),(1,2,4),(2,3,4)$ and terminal faces
\begin{equation}\label{eq:17faces}
\begin{gathered}
(0,1,5),(0,5,6);\qquad(1,2,7),(1,7,8);\\
(2,3,9),(2,9,10),(9,10,11);\\
(3,4,12),(3,12,13),(12,13,14);\\
(4,0,15),(4,15,16).
\end{gathered}
\end{equation}
The branch types and their ordered roots are
\[
 \begin{array}{c|ccccc}
 \text{type}&F_a&F_a&S_a&S_a&F_a\\
 \text{root}&(0,1)&(1,2)&(2,3)&(3,4)&(4,0).
 \end{array}
\]
For the explicit repair witness below, fix the serpent full forests with
final edges $\{9,11\}$ and $\{12,14\}$.

\begin{proposition}\label{prop:17}
For this graph $P(G)=4$ and $Z(G)=5$. For every minimum path cover a specified two-path cohomology class is nonzero.
\end{proposition}
\begin{proof}
For an induced forest $L$, write $x_e=\one_{e\in L}$, $d_v=\sum_{e\ni v}x_e$, and $f_\Delta=\sum_{e\subseteq\Delta}x_e$. Put
\[
 \mathcal T=\{(0,1,5),(0,5,6),(1,7,8),(9,10,11),(12,13,14)\}.
\]
Checking the coefficient of each graph edge gives
\begin{equation}\label{eq:17dual}
\begin{split}
 |L|={}&d_2+d_3+d_4+\sum_{\Delta\in\mathcal T}f_\Delta
            +x_{0,15}+x_{15,16}\\
       &-x_{0,5}-x_{2,3}-x_{2,4}-x_{3,4}.
\end{split}
\end{equation}
Using $d_v\le2$, $f_\Delta\le1$, $x_e\le1$ yields $|L|\le13$ and $P\ge4$.

If $|L|=13$, rewrite~\eqref{eq:17dual} as
\[
 13-|L|=\sum_{v=2,3,4}(2-d_v)
 +\sum_{\Delta\in\mathcal T}(1-f_\Delta)
 +(1-x_{0,15})+(1-x_{15,16})
 +x_{0,5}+x_{2,3}+x_{2,4}+x_{3,4}.
\]
Every summand is nonnegative. Hence all vanish. In particular
\[
 d_2=d_3=d_4=2,\quad f_\Delta=1\ (\Delta\in\mathcal T),\quad
 x_{0,15}=x_{15,16}=1,
\]
while the four subtracted selectors are zero. In face $(4,0,15)$ the selected edge $\{0,15\}$ forces $x_{0,4}=x_{4,15}=0$. In face $(4,15,16)$ the selected edge $\{15,16\}$ also forces $x_{4,16}=0$. The remaining neighbours available to vertex 4 are 1 and 12, so $d_4=2$ forces both $\{1,4\}$ and $\{4,12\}$. The faces $(0,1,4),(1,2,4),(3,4,12)$ then give $x_{0,1}=x_{1,2}=x_{3,12}=0$. Vertex 3 can now use only 9 and 13, so $d_3=2$ selects both, and faces $(2,3,9),(3,12,13)$ give $x_{2,9}=x_{12,13}=0$. Vertex 2 can then use only 7 and 10, so both are selected. Faces $(1,2,7),(2,9,10)$ force $x_{1,7}=x_{9,10}=0$.

Tightness of the face $(0,1,5)$ forces the edge $\{1,5\}$, since
$\{0,1\}$ and $\{0,5\}$ are absent. Vertex $1$ already lies on the
selected edges $\{1,4\}$ and $\{1,5\}$, so $\{1,8\}$ is absent.
The tight face $(1,7,8)$ therefore contributes $\{7,8\}$. The ten
compulsory edges are
\[
\begin{gathered}
 \{0,15\},\ \{15,16\},\ \{1,4\},\ \{4,12\},\ \{3,9\},\\
 \{3,13\},\ \{2,7\},\ \{2,10\},\ \{1,5\},\ \{7,8\}.
\end{gathered}
\]
The only remaining choices are one edge from each of
\begin{equation}\label{eq:17choices}
 \bigl\{\{0,6\},\{5,6\}\bigr\},\qquad
 \bigl\{\{9,11\},\{10,11\}\bigr\},\qquad
 \bigl\{\{12,14\},\{13,14\}\bigr\}.
\end{equation}
All eight choices satisfy the local forest conditions and have thirteen edges. They are exactly the four-path covers, so $P=4$.

Name the path containing $1,4$ by $A$ and order it with $1<4$. Name those containing $7,2$ and $9,3$ by $C,D$, with those respective orders. Every one of the eight covers has ordered subsegments
\[
 A:\ 1<4<12,\qquad C:\ 7<2<10,\qquad D:\ 9<3<13.
\]
The actual contact pairs $(1,7),(4,2)$ and $(4,3),(12,13)$ have agreeing orders, while $(2,3),(10,9)$ have opposite orders. Equation~\eqref{eq:contact-sign} therefore gives
\[
 s_As_C=1,\qquad s_As_D=1,\qquad s_Cs_D=-1.
\]
Their product is $1=-1$, impossible. Equivalently their ideal over $\CC$ contains 2, since
\begin{equation}\label{eq:17unit}
\begin{split}
2={}&(s_Cs_D+1)-s_As_D(s_As_C-1)-(s_As_D-1)\\
 &+s_Cs_D(s_A^2-1).
\end{split}
\end{equation}
This is a nonzero $H^1$ class on the triangle of path constraints $A,C,D$. Hence $Z\ge5$.

For an upper bound use paths
\[
 (5,1),\quad(4,12,14),\quad(6,0,15,16),\quad
 (10,2,7,8),\quad(11,9,3,13).
\]
Their starts give $S=\{4,5,6,10,11\}$. A chronology is
\[
\begin{gathered}
6\to0,\quad5\to1,\quad0\to15,\quad15\to16,\quad11\to9,\quad10\to2,\\
9\to3,\quad4\to12,\quad2\to7,\quad7\to8,\quad3\to13,\quad12\to14.
\end{gathered}
\]
Each source has only the indicated target white at that step. Thus $Z=5$.
\end{proof}
The selected core diagonal is $\{1,4\}$; the obstruction lies in the
contact orders between covering paths. Replacing any of the five branches
by a strip of the same ordered type preserves its conditional response,
and hence preserves $Z=5$. It also leaves the core and all demand sets
unchanged. The capacities at vertices $0,1,2,3,4$ remain $(1,1,0,0,1)$,
and $\{1,4\}$ is the only eligible core diagonal. Thus
Theorem~\ref{thm:path-capacity} preserves $\nu(C)=1$ and
$P=|V(C)|-\nu(C)=4$.

The repair formula also has two optimal choices here. The empty core
selection has $\tau=0$, whereas $J=\{\{1,4\}\}$ has $\tau=1$.
With $\{1,4\}$ retained, one possible designated cut is $\{2,7\}$.
An executable five-path cover is
\[
 (0,15,16),\quad(10,2),\quad(6,5,1,4,12,14),\quad
 (7,8),\quad(11,9,3,13).
\]
Thus omitting the core diagonal and repairing its phase obstruction have the same total cost.

\subsection{An all-fan family with a parity defect}
Take $k\ge3$. The core boundary is
\[
 a_0,b_0,c_0,a_1,b_1,c_1,\ldots,a_{k-1},b_{k-1},c_{k-1}.
\]
Use faces $(a_i,b_i,c_i)$ and $(c_{i-1},a_i,c_i)$, and triangulate the central polygon $c_0\cdots c_{k-1}$ arbitrarily. Add fans along all core outer sides: the fans on $a_ib_i$ and $c_{i-1}a_i$ are keyed at $a_i$, and that on $b_ic_i$ is keyed at $c_i$. All have at least two faces. Denote the graph by $A_k$.

\begin{proposition}\label{prop:all-fan}
\[
 P(A_k)=2k,\qquad Z(A_k)=2k+(k\bmod2).
\]
When all added fans have two faces, the order is $9k$.
\end{proposition}
\begin{proof}
The triangle $(a_i,b_i,c_i)$ with its two fans on $a_ib_i,b_ic_i$ is a neutral island on $a_ic_i$. Delete the $k$ disjoint island interiors by Theorem~\ref{thm:neutral}, paying exactly $k$ seeds. Around the central $k$-gon, the remaining block on $c_{i-1}c_i$ is $(c_{i-1},a_i,c_i)$ followed by the fan on $c_{i-1}a_i$ keyed at $a_i$. Viewed from $c_{i-1}c_i$, it retires both old roots and is a serpent. Theorem~\ref{thm:cyclic-strips} therefore gives
\[
 Z(A_k)=k+[k+(k\bmod2)].
\]
For $P$, the capacities are $c(a_i)=2,c(b_i)=0,c(c_i)=1$. Eligible core edges are arms $a_ic_i$ and edges between central $c$ vertices. If $h$ arms and $t$ central edges are selected, summing capacities at the $c_i$ gives $h+2t\le k$, hence $h+t\le k$. All $k$ arms are allowed, so $\nu=k$. The core has $3k$ vertices; Theorem~\ref{thm:path-capacity} gives $P=3k-k$.
\end{proof}
The parity obstruction is also visible in a hypothetical $2k$-seed
normal-form protocol. By Theorem~\ref{thm:hard-general}, choose such a
protocol satisfying the hard capacities. Let $h$ be its number of selected
arms, $t$ its number of selected central edges, and $W\ge0$ its total
terminal penalty. Then
\[
 h+2t\le k,\qquad 2k=3k-(h+t)+W.
\]
Hence $h+t-W=k$, and the two relations give $t+W\le0$. Thus
$t=W=0$ and $h=k$.

Let $X_i$ be the interior path of the fan on $c_{i-1}a_i$, ordered from
the neighbour $x_i$ of its nonkey $c_{i-1}$. The resulting core paths
contain the ordered segments
\[
 P_i=(a_i,c_i,X_{i+1}).
\]
The cross edges $c_{i-1}c_i$ and $x_i a_i$ have opposite orders on
$P_{i-1}$ and $P_i$, so~\eqref{eq:pair-sign} gives
$s_{i-1}s_i=-1$. Multiplication around the cycle yields $1=(-1)^k$,
which is impossible for odd $k$.

\subsection{Sharp common-edge gluing}
The following bound is the two-vertex instance of the vertex-cut inequality of~\cite[Proposition~3.5]{CJMZ}. Its proof is included to identify the boundary choices used in the sharp family.

\begin{proposition}\label{prop:gluing-bounds}
If finite simple graphs $H,Q$ meet exactly in the edge $ab$, then
\[
 Z(H)+Z(Q)-2\le Z(H\cup Q)\le Z(H)+Z(Q)+2.
\]
Put
\[
 \rho_H=\max\{|S\cap\{a,b\}|:S\text{ is a minimum zero forcing set of }H\},
\]
and define $\rho_Q$ similarly. Then the upper bound improves to
\[
 Z(H\cup Q)\le Z(H)+Z(Q)+2-\rho_H-\rho_Q.
\]
\end{proposition}
\begin{proof}
Restrict an optimal union arrow set to $H,Q$. The restrictions are executable, and
\[
 |M_H|+|M_Q|=|M|+\eps,\qquad
 |V(H)|+|V(Q)|=|V(H\cup Q)|+2,
\]
where $\eps$ records a selected common arrow. Subtraction proves the lower bound.

For the upper bound take optima in each block. Delete all $H$ arrows entering $a,b$ and all $Q$ arrows leaving them, at most four arrows in total. Both root columns in $H$ are now sources, and both root rows in $Q$ sinks. The union relation has no row-to-column passage. A column-only cycle could only come from $Q$ and a row-only cycle only from $H$, each impossible. Incoming occupations come only from $Q$ and outgoing ones only from $H$. Thus the union is executable. It has at least $|M_H|+|M_Q|-4$ arrows and two fewer vertices than the sum, proving the upper bound.

Choose the first optimum to attain $\rho_H$. Reverse a suitable optimum
in $Q$ so that $\rho_Q$ is its number of root terminals, using
Lemma~\ref{lem:deletion}. At most $(2-\rho_H)+(2-\rho_Q)$ arrows then need
deletion, giving the refinement.
\end{proof}

For $p,q\ge2$ with $\ell=p+q\ge5$, define a rooted graph $H_{p,q}$ on
\[
 a,b,v_0,\ldots,v_\ell,x_1,\ldots,x_\ell
\]
by faces
\[
 (a,v_{i-1},v_i)\ (1\le i\le p),\qquad
 (b,v_{i-1},v_i)\ (p+1\le i\le\ell),\qquad(a,v_p,b),
\]
and all ears $(v_{i-1},x_i,v_i)$. Its outer boundary is
$a,v_0,x_1,v_1,\ldots,x_\ell,v_\ell,b,a$.

\begin{proposition}\label{prop:sharp-plus-two}
\[
 P(H_{p,q})=Z(H_{p,q})=\ell-1.
\]
For two such blocks, with parameters $\ell$ and $m$ and disjoint interiors,
\[
 P(H_{p,q}\cup_{ab}H_{r,s})=Z(H_{p,q}\cup_{ab}H_{r,s})=\ell+m.
\]
Thus the positive correction $+2$ in Proposition~\ref{prop:gluing-bounds} is attained within maximal outerplanar graphs.
\end{proposition}
\begin{proof}
Every nonroot edge belongs either to exactly one ear triangle or to exactly one of the root degrees. The root is counted twice. Consequently every selected induced forest satisfies
\[
 |L|=d_L(a)+d_L(b)+\sum_i f_i-\one_{ab\in L}\le4+\ell,
\]
where $f_i$ is its edge count in the $i$th ear face. Since the order is $2\ell+3$, this gives $P\ge\ell-1$.

For $p=2,q=3$, take seeds $v_0,x_1,v_4,x_5$ and perform
\[
 x_1\to v_1,\ v_0\to a,\ x_5\to v_5,\ v_5\to b,
 \ a\to v_2,\ v_1\to x_2,\ b\to v_3,\ v_2\to x_3,\ v_3\to x_4.
\]
The two end spine vertices become black before their corresponding root
first forces. For a left extension, adjoin vertices $v_{-1},x_0$ and the
faces $(a,v_{-1},v_0)$ and $(v_{-1},x_0,v_0)$. Add $x_0$ to the initial
set. Immediately after $v_0$ becomes black, perform $x_0\to v_{-1}$ and
resume the old chronology; if $v_0$ is initial, perform this new force
before the old chronology starts. The only white neighbour of $x_0$ is
$v_{-1}$. The only old vertices with new neighbours are $a$ and $v_0$,
and all those neighbours are black before either vertex is used as a
source. Thus every subsequent old force remains legal. The new endpoint
$v_{-1}$ is black before $a$ first forces, preserving the induction
property. Relabelling the spine gives $H_{p+1,q}$. The right extension uses
$b,v_\ell$ instead of $a,v_0$ and gives $H_{p,q+1}$.

Starting from $H_{2,3}$ or its reflection $H_{3,2}$ reaches every
$p,q\ge2$ with $p+q\ge5$. Each extension adds one seed, so
$Z\le\ell-1$.

For the union, the same edge identity has $\ell+m$ ear faces and still only two root degrees. With $2\ell+2m+4$ vertices it gives $P\ge\ell+m$. Proposition~\ref{prop:gluing-bounds} gives
\[
 Z\le(\ell-1)+(m-1)+2=\ell+m.
\]
Both inequalities are equalities.
\end{proof}
For example two 13-vertex copies of $H_{2,3}$ have individual $P=Z=4$, whereas their 24-vertex union has $P=Z=10$.

\subsection{An eight-vertex leaf-count calculation}\label{subsec:eight}
Let the outer boundary be $0,1,\ldots,7$ and the bounded faces be
\[
 017,\quad137,\quad123,\quad357,\quad345,\quad567.
\]
The weak-dual degrees, in this order, are $(1,3,1,3,1,1)$, so it has four
leaves. The initial set and forces
\[
 S=\{0,1,6\},\qquad
 0\to7,\ 6\to5,\ 7\to3,\ 1\to2,\ 3\to4
\]
give a three-seed chronology. For a lower bound the only degree-two vertices are $0,2,4,6$; all other degrees are at least four. A two-seed process can start only at one of these ears together with one adjacent seed. After its first force precisely the ear triangle is black. Each non-ear vertex of that triangle still has at least two white neighbours, and the ear has none, so the process stops. Hence $Z=3$.

Thus $Z(G)=3<4=\ell$, contradicting the weak-dual-leaf lower bound
stated as Theorem~1 on page~3 of the 2023 preprint version of~\cite{IKKpre}. In the cover
\[
 (0,7,3,4),\quad(6,5),\quad(1,2),
\]
leaf faces $017,123$ share the endpoint 1, and $345,567$ share the endpoint 5. Counting two endpoints separately at each leaf therefore counts the same endpoint more than once. 

\appendix
\section{Exact algorithms and finite certificates}\label{app:algorithms}

We give the six-port recurrence, the capacity recurrence, and the finite
identities used in the proofs. The recurrences also recover a minimum
zero forcing set and its chronology.

\subsection{The six-port operation}\label{app:sixports}

Let the root triangle be $acb$. The ordered child roots are $(a,c)$ and $(c,b)$.
A child state consists of four occupation bits and a strict reachability
relation on its four ports. Write the left occupations as
$(i_a^L,i_c^L,o_a^L,o_c^L)$ and the right occupations as
$(i_c^R,i_b^R,o_c^R,o_b^R)$. For a root choice
$\eta\in\{0,+,-\}$, where $+$ means $a\to b$ and $-$ means $b\to a$, form
\begin{align}
 i_a&=i_a^L+\one_{\eta=-},&o_a&=o_a^L+\one_{\eta=+},\notag\\
 i_c&=i_c^L+i_c^R,&o_c&=o_c^L+o_c^R,\notag\\
 i_b&=i_b^R+\one_{\eta=+},&o_b&=o_b^R+\one_{\eta=-}.
 \label{eq:app-occupations}
\end{align}
Reject a choice if any of these six integers exceeds one.

Embed the child relations in the port set
\[
 p_a,q_a,p_c,q_c,p_b,q_b.
\]
Adjoin $q_a\to p_b$ if $\eta=+$, and $p_b\to q_a$ otherwise. Adjoin
$q_b\to p_a$ if $\eta=-$, and $p_a\to q_b$ otherwise. Take the
positive-length transitive closure. Reject the choice if a diagonal entry
is present; otherwise restrict the closure to $p_a,q_a,p_b,q_b$.
Together with the four exterior occupations, this is
$\Phi(s_L,s_R,\eta)$. The order of the last two steps is essential: a cycle
through $p_c$ or $q_c$ must be detected before these ports are forgotten.

The seven possible relation shapes and their occupation conditions generate
exactly 56 states. Applying the preceding operation to all $56^2\cdot3$
triples yields 1,410 allowed triples.

For a rooted block, let $m_H(s)$ be the maximum number of selected arrows
with state $s$, and give an unrealizable state value $-\infty$. On the
single-edge block the empty choice has value zero and the two directed
choices have value one. At a triangle,
\begin{equation}\label{eq:app-max-arrows}
 m_H(s)=\max_{\Phi(u,v,\eta)=s}
 \left(m_L(u)+m_R(v)+\one_{\eta\ne0}\right).
\end{equation}
Equivalently, for $r_H(s)=|V(H)|-2-m_H(s)$,
\begin{equation}\label{eq:app-raw-min}
 r_H(s)=\min_{\Phi(u,v,\eta)=s}
 \left(r_L(u)+r_R(v)+1-\one_{\eta\ne0}\right).
\end{equation}
The minimum is taken over a finite list of explicitly generated transitions.
Store one attaining triple with every finite entry. Traceback returns an
attaining arrow set.

\subsection{A minimum set and its chronological certificate}
\label{app:constructive-witness}

Generate the 56 states and 1,410 transitions by the six-port rule. At
each face and attainable state, store the maximum in
\eqref{eq:app-max-arrows} and one attaining triple. These choices recover
an optimal arrow set, whose incidence digraph determines a chronology.

\begin{proposition}\label{prop:constructive-algorithm}
Let $G$ be a polygon triangulation on $n\ge3$ vertices, supplied with indexed
face adjacency and an ordered boundary edge. A minimum zero forcing set $S$
and a chronological sequence of its $n-|S|$ forces can be constructed in
$O(n)$ word-RAM time and $O(n)$ words of storage. The words have $O(\log n)$
bits. A supplied initial set and chronology can be checked against the
original color rule in $O(n)$ time, independently of the state tables.
\end{proposition}
\begin{proof}
The rooted face tree has $n-2$ internal nodes and $n-1$ single-edge leaves.
Process it in postorder. Each internal node tests the fixed list of 1,410
transitions, with integers between zero and $n-1$ and one pointer per
attainable state. The exact interface criterion and
\eqref{eq:app-max-arrows} show inductively that the stored entry is the
largest number of arrows realizing that state. At the root choose a state
attaining $\max_s m_G(s)$. Follow the stored triples to the leaves. Every
original edge is the new root edge of exactly one node of this decomposition;
hence its arrow is reported once. The recovered arrow set $M$ satisfies
\[
 |M|=\max_s m_G(s),\qquad
 S=V\setminus\{v:\exists u,\ u\to v\in M\},\qquad
 |S|=n-|M|=Z(G).
\]
There are $O(n)$ entries and pointers because the number of states is fixed.
Both the evaluation and traceback therefore take $O(n)$ operations.

To order the arrows, form the incidence digraph $D_G(M)$ of
\eqref{eq:incidence}. It has $2n$ vertices and exactly
\begin{equation}\label{eq:sparse-incidence-size}
 n+2|E(G)|=n+2(2n-3)=5n-6
\end{equation}
arcs. The traceback realizes an executable state, so this digraph is acyclic.
Compute a topological order by a queue of zero-indegree vertices. For each
noninitial column $p_v$, emit its unique selected predecessor arrow $u\to v$
when $p_v$ is removed from the queue. Initially color all of $S$ black. Every
required neighbor $w\in N[u]\setminus\{v\}$ satisfies
\[
 p_w\longrightarrow q_u\longrightarrow p_v.
\]
It has therefore either been initial or has already appeared as the target
of an emitted force. Consequently $u$ is black and $v$ is its only white
neighbor. The emitted order is a legal chronology. Equation
\eqref{eq:sparse-incidence-size} gives linear time and storage for this step.

For the independent check, maintain at each vertex the number of its white
neighbors and the bitwise exclusive-or of their integer indices. Initialize
these quantities by scanning the adjacency lists once. A proposed $u\to v$
is valid precisely when $u$ is black, $v$ is white, the count at $u$ is one,
and its exclusive-or equals the index of $v$. On coloring $v$, update these
two quantities at every neighbor of $v$. Each vertex changes color once and
each adjacency is updated at most once in this way, so the total work is
$O(n+|E(G)|)=O(n)$.
\end{proof}

The program \texttt{algorithms/constructive\_zero\_forcing.py} implements
this construction using the Python standard library. Its JSON input
contains the triangular faces and an optional ordered root; its output
contains $Z(G)$, $S$, the selected arrows, and their chronology. The
56 states and 1,410 transitions are generated at startup.

The face decomposition uses a stack, and the incidence graph is processed
by a queue. The JSON reader indexes face incidences by dictionaries;
the deterministic complexity in Proposition~\ref{prop:constructive-algorithm}
is for indexed face adjacency.

\begin{verbatim}
python algorithms/constructive_zero_forcing.py graph.json -o solution.json
\end{verbatim}

For the eight-vertex triangulation in Section~\ref{subsec:eight}, one output is
\[
 S=\{0,1,6\},\qquad
 0\to7,\quad6\to5,\quad7\to3,\quad5\to4,\quad3\to2.
\]

\subsection{Core capacities, ears and traceback}\label{app:capacity-algorithm}

The capacity recursion stores labels
\[
 (s,j_a,j_b,k_a,k_b).
\]
The integers $j_a,j_b$ count selected \emph{core} edges; the bits $k_a,k_b$
are credits from the already present boundary sides. The state $s$
contains the total incoming and outgoing occupations.

For a terminal strip, the root is unselected and $j_a=j_b=0$. Its states
are the canonical modes in Section~\ref{sec:strips}, with their listed
costs. A fan keyed at $a$ has credits $(1,0)$, a fan keyed at $b$ has
$(0,1)$, and a serpent has $(0,0)$. For a triangular ear, a mode using
$a$ has credits $(0,1)$ and a mode using $b$ has $(1,0)$. Its empty mode
is entered with both possible credit pairs. Thus the ear allocation is
optimized as explicit finite data.

A bare boundary edge is also a leaf of the core recursion. Its three
states have
\begin{equation}\label{eq:app-bare-leaf}
 (j_a,j_b)=(\eps,\eps),\qquad
 (k_a,k_b)=(1,1),\qquad r=-\eps,
\end{equation}
where $\eps$ is zero or one according to whether that edge is selected.
The credits express the absence of a terminal demand on that side. Bare
edges remain eligible core edges.

At a core triangle combine a left label and a right label only when
\begin{equation}\label{eq:app-forgotten-capacity}
 j_c^L+j_c^R\le k_c^L+k_c^R.
\end{equation}
For every allowed $\Phi$ transition, the output has
\begin{equation}\label{eq:app-capacity-output}
 (j_a,j_b,k_a,k_b)
 =\left(j_a^L+\eps,j_b^R+\eps,k_a^L,k_b^R\right)
\end{equation}
and cost $r_L+r_R+1-\eps$. Keep the least cost at each label.
Within one credit pair, delete a weighted label only if a distinct one
strictly dominates it in the sense of Lemma~\ref{lem:dominance}.
The frontier preserves every future feasible minimum by
Lemma~\ref{lem:dominance}.

Choose one actual terminal root of the canonical core as the final cut.
The inside table includes all other sides; the final terminal table is
kept on the other side. When closing the root, require identical root
choices, compatible occupations and acyclic united relations, together
with
\[
 j_a^L+j_a^R-\eps\le k_a^L+k_a^R,
 \qquad
 j_b^L+j_b^R-\eps\le k_b^L+k_b^R.
\]
The final cost is $2+r_L+r_R+\eps$, with the shared root arrow counted
once. For a terminal root, $\eps=0$.

Every retained entry stores its realizing child labels and root choice.
Traceback determines the ear demands, core arrows and terminal modes.
Realize each mode in the original strip by the explicit root-forcing
construction, reversal and deletion. Maintain the current black set $B$,
initially the complement of the target set. A pending arrow $u\to v$ is
legal exactly when
\[
 u\in B,\qquad v\notin B,\qquad N(u)\setminus B=\{v\}.
\]
After performing it, add $v$ to $B$. This execution check uses the original
color rule.

\subsection{Direct evaluation of the capacity--phase formula}
\label{app:literal-phase}

The formula of Section~\ref{sec:phases} can also be evaluated by
enumerating core selections, designated cuts, and path directions. This
procedure has exponential running time.

Given the triangular faces, reconstruct the weak dual and its minimal
subtree containing all degree-three vertices. If that subtree is empty,
the answer is two. Otherwise construct the core and every actual terminal
strip. Enumerate the ear choices $\theta$, then the subsets
$J\subseteq E(C)\setminus\cR$ satisfying the vertex and face capacities.
For each such pair, construct the fixed full terminal forests by the
original color rule. Record their selected root-to-interior edges
$\cA_\theta$.

For every $D\subseteq\cA_\theta$, form $L(J,D)$ and recover its components
in actual path order. Verify that they are induced, vertex-disjoint and
cover all vertices, and that their number is
\[
 r-|J|+|D|.
\]
Enumerate directions of the nonsingleton paths and execute precisely
the corresponding path arrows. A successful execution is a point of the
phase scheme; a failed execution can instead be certified by a dependency
cycle or by a blocking circuit. The minimum successful path count is the
right side of the main formula.

Process candidates in increasing total cost and stop at the first
feasible cost. Only nontrivial paths receive sign variables.

\subsection{Finite verification data}\label{app:finite-proof}
\label{app:closure-contract}

The following finite data and integer identities enter the proofs of
Lemmas~\ref{lem:568}, \ref{lem:3912}, and the replacement results of
Section~\ref{sec:replacement}.
\begin{center}\small
\begin{tabular}{@{}>{\raggedright\arraybackslash}p{26mm}p{116mm}@{}}
\toprule
Data&Finite identity\\\midrule
56 causal states&
Reconstruct the states, their triangle witnesses, all 1,410 allowed
six-port transitions, and all 387 compatible shared-root pairs.\\[3pt]
568 responses&
Evaluate every recursive polygon witness and all $568^2=322,624$
products; compare all coordinates and scalar shifts.\\[3pt]
3,912 joint data&
Evaluate every partial-core witness, all $3,912^2=15,303,744$ products,
and all $4\cdot3,912=15,648$ closing equalities.\\[3pt]
6 terminal responses&
Check all 672 coordinates of the unary table and both min-plus
reciprocity identities.\\[3pt]
209 corridor operators&
Evaluate every witness word, all 418 left-generator products, and the
root endpoint identifications; representatives have length at most ten.\\[3pt]
284 context classes&
Compute all 322,624 conditional-type pairings and 80,656 quotient
products. Seventeen polygon contexts separate all 40,186 distinct class
pairs; each context has at most 15 vertices.\\
\bottomrule
\end{tabular}
\end{center}

For the 568-type certificate, the numbers of products with shifts
$0,1,2,3$ are respectively $8,139$, $151,699$, $151,761$, and $11,025$.
The largest representative has 305 faces. For the capacity certificate,
the resulting frontiers have at most thirty entries, normalized costs
in $\{0,1,2,3,4\}$, offsets in $\{0,1\}$, and removed minima in
$\{0,1,2,3\}$.

The capacity product is computed as follows. Initialize an empty table.
For every pair of finite child labels and every $\eta\in\{0,+,-\}$,
apply~\eqref{eq:forget-capacity} and the six-port rule. Reject a failed
occupation or cycle test; otherwise compute the label
\eqref{eq:capacity-update} and cost~\eqref{eq:capacity-cost}. At each
label keep the least value and one attaining pair. After examining all
candidates, remove only the strictly dominated weighted labels of
Lemma~\ref{lem:dominance}. Subtract the finite minimum $u$, update the
unrestricted type and offset by~\eqref{eq:joint-product}, and return the
normalized datum together with $u$. Closing checks
\eqref{eq:closing-capacity} and causal compatibility before taking either
minimum. The stored values are the minima of these finite candidate sets.

The primitive data represent the leaf blocks. At a triangle, the
interface criterion identifies every parent realization with a compatible
pair of child realizations and a root choice. The product therefore
computes its weighted datum. Closure supplies a listed datum for every
recursive term, and witness evaluation identifies that datum with the
same graph. At the final root, the verified closing equality gives the
equality of the restricted and unrestricted optima. Induction on the number of faces therefore proves equality for every
decorated core.

Each certificate witness is a primitive or a pair of earlier data, so
its recursion is acyclic. The unrestricted input is checked independently
by reconstructing the local incidence data and recomputing all 322,624
products before the capacity comparison. The input parser requires the
specified number of distinct data, all primitive keys, valid state and
type indices, nonempty normalized frontiers, and sorted, distinct frontier
labels. It rejects truncated data and trailing tokens. Infinity denotes
an absent candidate: such entries are skipped before finite arithmetic,
and are never normalized by subtraction. The finite values are below the
integer sentinel used by the implementation.

\subsection{Programs}\label{app:reproducibility}

The accompanying source contains the finite data and their verifiers.
The combinatorial programs use Python~3.10 or later and its standard
library. The complete capacity comparison uses a C++17 verifier, which
evaluates all $15,303,744$ ordered products and $15,648$ closing
equalities. The complete verification driver also uses the dependencies
listed in \texttt{requirements.txt}. From the source directory, run
\begin{verbatim}
python verify.py --full
\end{verbatim}
The driver evaluates the local states, the 568-type response algebra,
and the capacity comparison in that order. The source directory also
contains implementations of the terminal, corridor, and context
calculations in Section~\ref{sec:replacement}.

\section*{Acknowledgments}
We began studying zero forcing in 2021, during our second year as
undergraduates, through a combinatorial argument for maximal biconnected
outerplanar graphs. We later returned to that argument, recast it in terms
of sheaf cohomology, and pursued its extension to general finite graphs.
This research grew out of our own mathematical curiosity and our interest
in exploring these questions.

We used AI tools to assist with mathematical derivations, computations,
and preliminary drafts. We are now checking the arguments and rewriting
both manuscripts ourselves. The versions submitted for publication will
be the manuscripts rewritten by the authors.

\end{document}